\documentclass[letterpaper,twoside,11pt,leqno]{article}
\usepackage{amsmath,amsthm,amssymb, enumerate}
            \newcommand{\marginalnote}[1]{}
\usepackage[all]{xy}
\usepackage{graphicx}
\usepackage{epsfig}
\usepackage{color}	
\usepackage{ulem}	

\theoremstyle{plain}

\newtheorem{Thm}{Theorem}[section]
\newtheorem{Prop}[Thm]{Proposition}
\newtheorem{Lem}[Thm]{Lemma}
\newtheorem{Cor}[Thm]{Corollary}

\newtheorem*{Thm*}{Theorem}
\theoremstyle{remark}

\theoremstyle{definition}
\newtheorem{Rem}[Thm]{Remark}
\newtheorem{Ex}[Thm]{Example}

\newtheorem{Not}[Thm]{Notation}
\newtheorem{Def}[Thm]{Definition}
\newtheorem{Defs}[Thm]{Definitions}

\newtheorem{Que}[Thm]{Question}

\newcommand{\gen}[1]{\left\langle#1\right\rangle}

\newcommand{\gp}[2]{\gen{#1\mid #2}}

\newcommand{\Z}{\mathbb{Z}}

\newcommand{\ga}{\Gamma}
\newcommand{\word}{w}

\DeclareMathOperator{\supp}{supp}
\DeclareMathOperator{\Star}{St}
\DeclareMathOperator{\Link}{Lk}
\DeclareMathOperator{\KLink}{K-Lk}

\def\coloneqq{\mathrel{\mathop\mathchar"303A}\mkern-1.2mu=}

\begin{document}
\date{}
\author{Yago Antol\'{i}n  and Laura Ciobanu}
\title{Geodesic growth in right-angled and even Coxeter groups}
\maketitle

\begin{abstract}
The objective of this paper is to detect which combinatorial properties of a regular graph can completely determine the geodesic growth of the right-angled Coxeter or Artin group this graph defines, and to provide the first examples of right-angled and even Coxeter groups with the same geodesic growth series. 

\bigskip


\end{abstract}

\section{Introduction}

The \textit{geodesic growth function} of a group $G$ with respect to a finite generating set $S$ counts, for each positive integer $n$, the number of geodesics of length $n$ starting at $1$ in the Cayley graph of $G$ with respect to $S$. The \textit{geodesic growth series} is the formal series that takes the values of the geodesic growth function as series coefficients (see Definition \ref{defgrowth}).

One of the first beautiful results in this area states that the geodesic growth series of a hyperbolic group, with respect to any finite generating set, is a rational function. This follows from the fact that the language of geodesics is regular \cite{Epsteinetc}. Such behaviour also appears in other important classes of groups: Artin groups of large (\cite{HoltRees}) and spherical (\cite{Charney}) type,  Coxeter groups (\cite{BjornerBrenti}), Garside groups over the Garside generators (\cite{CharneyMeier}), or orientable surface groups of finite genus (\cite{GN}), to name just a few. However, unlike standard (spherical) group growth, geodesic growth is highly sensitive to the generating set. By an example due to Cannon \cite[Remark]{neushap}, there is a group $G$ and generating sets $S$ and $S'$ such that $G$ has a regular set of geodesics with respect to $S$, and non-regular with respect to $S'$. More recently, some inroads have been made into understanding those groups that posses generating sets for which the geodesic growth is polynomial \cite{Bridsonetc}. 

 The groups that we treat in this paper, right-angled and even Coxeter and right-angled Artin groups, are known to have a regular language of geodesics, and therefore rational geodesic growth series, with respect to the standard generating sets (see \cite{JJJ} or \cite[Theorem 4.8.3]{BjornerBrenti}, for proofs of these facts). 

Our goal here is to give more qualitative information on the geodesic growth. More precisely, this paper was motivated by the following Question: 

\begin{Que}\cite[Question 1]{JJJ} Can two different right-angled Artin groups have the same geodesic growth series?\end{Que}
 The answer, stated in Corollaries \ref{cor:sameGGC} and \ref{cor:sameGGA}, is positive:

\medskip

 {\it Let $G_1$ and $G_2$ be two right-angled Artin or Coxeter groups based on link-regular graphs that have the same $f$-polynomial and the same size link for each clique of a given dimension. Then $G_1$ and $G_2$ have the same geodesic growth series.}

\medskip

This is one of our main results, and it is proved in Section \ref{MT}. There we analyze the automata that recognize the geodesic languages of the two given groups and make use of their properties to reach the criteria in the hypothesis. This result and its proof was inspired by an analogous statement for automatic growth of right-angled Coxeter groups in \cite[Theorem 14]{GS2009}, where a different definition of link-regularity is used. In Section \ref{relation} we prove that the two definitions of link-regularity are in fact equivalent. 

When one sets the additional condition that the graphs that determine the groups be triangle-free, the result also follows from employing the growth formula for languages defined by `forbidden words', as introduced in \cite{GN}. This line of proof can then be generalized to even Coxeter groups, which is the topic of Section \ref{evenCoxeter}. The rigid chains of forbidden words for triangle-free even Coxeter groups are described in Lemmas \ref{Lem:rigidsequence1} and \ref{Lem:rigidsequence2} and this characterization is used to obtain the main result of Section \ref{evenCoxeter}:

\medskip

\textbf{Theorem 6.15}.
{\it Let $(G,S)$ be an even Coxeter system with graph $\Gamma.$
Suppose that $\Gamma$ is triangle-free,  and star-regular, that is, the stars of any two vertices in $\Gamma$ are isomorphic as graphs.
Then the geodesic growth function depends only on $|S|$ and the isomorphism class of the star of the vertices.

In particular, if $(G_1,S_1)$ and $(G_2,S_2)$ are triangle-free, star-regular, even Coxeter systems with $|S_1|=|S_2|$ and $\Star(v)\cong \Star(u)$, $v\in V\Gamma_1$, $u\in V\Gamma_2$, then $G_1$ and $G_2$ have the same geodesic growth. }

\medskip

While we found sufficient conditions for two right-angled Artin or Coxeter groups to have the same geodesic growth, we have not found necessary conditions for this to happen. For example, it is not clear whether a pair of groups based on two non-isomorphic trees can have the same geodesic or same automatic growth. We discuss some of the connections between the spherical, automatic and geodesic growth of right-angled Coxeter groups in Section \ref{relation}, and give a formula for the geodesic growth of the key players in this paper, the right-angled Coxeter groups based on regular triangle-free graphs, in Section \ref{sec:formula}.

The Appendix contains a description of the geodesic words that represent elements in the centralizers of generators in even Coxeter groups. This is used in the proof of Theorem \ref{ThmevenCoxeter}, and is basically a result of Bahls and Mihalik \cite{BahlsMihalik}. 


\section{Definitions and notation}

Let $S$ be a finite set and $S^*$ the free monoid on $S$. We indentify $S^*$ with the set of words over $S$, that is, finite sequences of elements of $S$, and we use $\equiv$ to denote equality of words. If $w_1, w_2$ are words over $S$, denote by $(w_1,w_2)$ the word obtained by concatenating $w_1$ and $w_2$.

Let $G=\gen{S}$ be a group generated by $S$. For an element $g$ of $G$, denote by $|g|\, (=|g|_S)$ the word length of $g$ with respect to $S$, and by $\pi_S\colon S^* \rightarrow G$ the natural projection.

A {\it Coxeter system} $(G,S)$ is a pair consisting of a group $G$ together with a distinguished generating set $S=\{s_i\}_{i\in I},$ for which there is a presentation 
$$\langle S \mid\mid (s_is_j)^{m_{i,j}}=1, \quad i,j\in I\rangle,$$
where $m_{i,j}\in \{0,1,2,3,\dots\}$ if $i\neq j$, $m_{i,i}=1$, (i.e. $s_i$ is an involution) and $m_{i,j}=m_{j,i}$. Here, we understand that $g^0=1.$   The Coxeter system is {\it even} (resp. {\it right-angled}) if the numbers $m_{i,j}$ are even for $i\neq j$ (resp. equal to 2 or $0$).
The group $G$ is called a {\it Coxeter group}, {\it even Coxeter group} or {\it right-angled Coxeter group} if there is a subset $S$ of $G$ such that $(G,S)$ is, respectively, a Coxeter system, even Coxeter system or a right-angled Coxeter system.

Given a Coxeter system $(G,S)$, let $\ga=\ga(G,S)$ be a simplicial labelled graph with vertex set $S=V\Gamma$ and edge set $E=E\Gamma$, where $\{s_i,s_j\}\in E\ga$ if  $m_{i,j}\neq0.$ The vertices are labelled by the elements of $S$ and the edges by the numbers $m_{i,j}.$ The Coxeter system $(G,S)$ is {\it triangle-free} if the graph $\ga(G,S)$ is triangle-free, that is, no three vertices of $\ga$ span a complete graph.

In the case of \textit{right-angled Coxeter groups} (racg's), we will omit the labels on the edges of $\ga$, since all of them are 2. The racg determined by $\Gamma$ is given by the presentation 
$$\langle s \in S \mid\mid s^2=1 \ \forall s\in S, \ \textrm{and}\  (ss')^2=1 \ \forall \{s, s'\} \in E \rangle.$$
If we don't require the generators to be involutions, we obtain the related \textit{right-angled Artin groups} (raag's). The raag determined by a graph $\Gamma$ is given by the presentation $$\langle s \in S \mid\mid ss'=s's   \ \forall \{s, s'\} \in E\Gamma \rangle.$$

For a Coxeter system $(G,S)$ and $\ga=\ga(G,S)$, we will use the same letters for vertices of $\ga$ and the corresponding generators.

Let $\mathcal{A}$ denote the set of non-empty subsets  $\sigma \subseteq V\ga$, with the property that the vertices of $\sigma$ span a complete subgraph or \textit{clique} of $\Gamma.$ The {\it size of } $\sigma\in \mathcal{A}$ is just $|\sigma|$.
For any $\sigma \in \mathcal{A}$,  we let $\Link(\sigma)$ denote the vertices in $V\Gamma\setminus \sigma$ that are connected with every vertex in $\sigma$. That is, 
$$\Link(\sigma)=\{v \in V\Gamma \setminus \sigma : \{v\} \cup \sigma \in \mathcal{A}\},$$
and we let $\Star(\sigma)$ denote the vertices in $\Gamma$ that are connected with every vertex in $\sigma$. That is, 
$$\Star(\sigma)=\{v \in V\Gamma : \{v\} \cup \sigma \in \mathcal{A}\}.$$

\begin{Defs} \label{defgrowth}
Let $G$ be a group generated by $S$ and let $\pi\colon S^*\to G$ be the natural projection. We define the following.
\begin{enumerate}
\item The \textit{geodesic growth function} $\gamma_{(G,S)}:\mathbb{N}\rightarrow \mathbb{N}$ is given by 
$$\gamma_{(G,S)}(r)=|\{w\in S^* : |w|=|\pi(w)|=r\}|,$$
and the \textit{geodesic growth series} of $G$ is given by $$\mathcal{G}_{(G,S)}(z)=\sum_{i=0}^\infty \gamma_{(G,S)}(i)z^i.$$ 
\item The \textit{spherical growth function} $\sigma_{(G,S)}:\mathbb{N}\rightarrow \mathbb{N}$ is given by 
$$\sigma_{(G,S)}(r)=|\{g\in G : |g|=r\}|,$$
and the \textit{spherical growth series} of $G$ is given by $$\Sigma_{(G,S)}(z)=\Sigma_S(z)=\sum_{i=0}^\infty \sigma_{(G,S)}(i)z^i.$$ 
\end{enumerate}
More generally, let $\mathcal{L}$ be a language over a finite alphabet $S$. 
By $\mathcal{L}(z)$ we denote the characteristic  formal power series of the language
$$\mathcal{L}(z)=\sum_{w\in \mathcal{L}}z^{|w|},$$
where $|w|$ is the length of $w$ as a word on $S$.\end{Defs}
 
We recall that the {\it $f$-polynomial} of a graph $\Gamma$ is the generating function for the number of cliques of size $i+1$ in $\Gamma$: that is, $f(t)=1+f_0t+f_1t^2+ \dots $, where $f_i$ is the number of $(i+1)$-cliques in $\Gamma$. The $f$-polynomial of a simplicial complex $K$ is the generating function for the numbers of $i$-dimensional simplices of $K$: that is,  $f(t)=1+f_0t+f_1t^2+ \dots $, where $f_i$ is the number of $i$-simplices in $K$.

\section{Main Theorem}\label{MT}

\begin{Def} \label{link-regular}
A graph $\Gamma$ is called \textit{link-regular} if for all $\sigma \in \mathcal{A}$, $|\Link(\sigma)|$ depends on $|\sigma|$ and not on $\sigma$ itself. If $\ga$ is link-regular, we will write $|\Link(i)|$ to denote $|\Link(\sigma)|$ with $|\sigma|=i$.
\end{Def}

Now denote by $\mathcal{D}$ the deterministic finite automaton accepting the geodesic language of a right-angled Coxeter group, as described in \cite{JJJ}. See \cite{hu} for definitions and more information on languages and automata.

\begin{Prop}\label{prop:automaton}
Let $G$ be a right-angled Coxeter group based on a simplicial graph $\Gamma$. Then $\mathcal{D}$ has states $\mathcal{S}=\mathcal{A}\cup \{\emptyset, \rho\}$, where $\emptyset$ is the start state, and $\rho$ is a single fail state. The accept states are $\mathcal{A}\cup \{\emptyset \}$ and the transition function $\mu:\mathcal{S} \times V\Gamma \rightarrow \mathcal{S}$ is given by
\begin{enumerate}
\item $\mu(\sigma, v)=\{v\} \cup (\Star(v)\cap \sigma)$ for all $\sigma  \in \mathcal{A}\cup \{\emptyset\}$ and $v \notin \sigma$ and
\item $\mu(\sigma,v)=\rho$ otherwise.
\end{enumerate} 
\end{Prop}

\begin{proof}

We rely on the notation in the proof of \cite[Theorem 1]{JJJ}. Suppose that $|V\Gamma|=n$ and that $G$ has generators $\{s_1, \dots, s_n\}$. For each $v_i\in V\Gamma$ (we exceptionally distinguish between generators and vertices in this case in order to make the exposition clearer) we consider the automaton $F_i$ recognizing the geodesic language of the group $(G_i, A_i)$, where $G_i=\langle s_i \mid s_i^2=1 \rangle$ and $A_i=\{s_i\}$, consists of two accepting states: the start state $1_i$, $v_i$, and one fail state $\rho_i$. There is only one transition between the accepting states: an $s_i$-edge from $1_i$ to $v_i$. The modification $\widehat{F_i}$ of $F_i$ has the same failing state, and the same two accepting states and the $s_i$-edge between them, as well as $s_j$-loops at $1_i$ for all $j \neq i$,  $s_j$-loops at $v_i$ for all $[v_j, v_i]=1$, and $s_j$-edges from $v_i$ to $1_i$ for all $[v_j, v_i]\neq1$. 

The automaton $\mathcal{D}$ is the product of the automata $\widehat{F_i}$, $i=1,\dots, n$. The states of $\mathcal{D}$ are $n$-tuples $(p_1,\dots, p_n)$, where $p_i$ is a state in $\widehat{F_i}$, the start state corresponds to $(1_1, \dots, 1_n)$, the failing states correspond to those $n$-tuples that contain a  coordinate with a failing state, and finally, there is an $s_j$-edge from $(p_1,\dots, p_n)$ to $(q_1,\dots, q_n)$ if there is an $s_j$-edge from $p_i$ to $q_i$ in $\widehat{F}_i$ for every $i$.

We now show that $\mathcal{D}$ coincides with the automaton described in the proposition. The non-fail states of $\mathcal{D}$ are by definition $n$-tuples $(p_1, \dots, p_n)$, where $p_i \in F_i\setminus \rho_i$, $1\leq i \leq n$, that is $p_i$ is either $1_i$ or $v_i$.  Then, a non-fail state is codified by a subset $\sigma(p_1,\dots,p_n)$ of $V\Gamma$ given by $\sigma(p_1,\dots,p_n)=\{v_i\in V\Gamma : p_i=v_i\}$, with $\emptyset$ corresponding to $(1_1, \dots, 1_n)$.

By the definition of $\mathcal{D}$ one can check that $\mu(\sigma, v)=\{v\} \cup (\Star(v)\cap \sigma)$ if $v \notin \sigma$ and $\mu(\sigma,v)=\rho$ otherwise. Thus all transitions starting at state corresponding to an element $\sigma$ of $\mathcal{A}$ end in a state corresponding to $\{v\} \cup (\Star(v)\cap \sigma)$, and as  all the vertices of $\Star(v)\cap \sigma$  commute with $v$, this state corresponds to an element of $\mathcal{A}$ as well. Thus if an $n$-tuple $(p_1, \dots, p_n)$ does not correspond to an element of $\mathcal{A}$, there is no path from $\emptyset$ to $(p_1, \dots, p_n)$, so these tuples correspond to states that are in fact not accessible.
\end{proof}

\begin{Defs}
Let $\sigma\in \mathcal{A}$, $\tau\subset \sigma$ and $j\in \mathbb{N}$.
We define  $$\Link_j(\sigma)=\{v\in V\Gamma\setminus \sigma: |\Star(v)\cap \sigma|=j\},$$ i.e. the set of vertices in $\Gamma \setminus \sigma$ that commute with exactly $j$ vertices in $\sigma$, and $$\Link_{\tau}(\sigma)=\{v \in \Gamma \setminus \sigma : \Star(v)\cap\sigma =\tau \}$$ i.e. the set of vertices outside of $\sigma$ that commute with all the vertices in $\tau$ and no vertex in $\sigma\setminus \tau$.

Let $\deg_{\tau}(\sigma)=|\Link_{\tau}(\sigma)|$ and $\deg_j(\sigma)=|\Link_j(\sigma)|$.
\end{Defs}

\begin{Lem}\label{lem:degj}
Let $G$ be a right-angled Coxeter group based on a link-regular simplicial graph $\Gamma$. 

The value $\deg_j(\sigma)$, $\sigma \in \mathcal{A}$, $j< |\sigma|$,  depends on $|\sigma|$, $j$ and the values $|\Link(k)|$ for $k=1,2,\dots$, and not on $\sigma$ itself.
\end{Lem}

\begin{proof}
Let $\sigma \in \mathcal{A}$ and  $i=|\sigma|$.

\begin{equation}\label{eq:disjoint}
\text{If }\tau_1,\tau_2\subseteq \sigma, \, \tau_1\neq\tau_2 \text{ then } \Link_{\tau_1}(\sigma) \cap \Link_{\tau_2}(\sigma)=\emptyset.
\end{equation}
Indeed, without loss of generality assume that $|\tau_1|\leq|\tau_2|$, If $v \in \Link_{\tau_1}(\sigma) \cap \Link_{\tau_2}(\sigma)$, then $v$ commutes with at least one more vertex from $\sigma$ than those in $\tau_1$, which gives a contradiction.

It follows from the definition of $\Link_j$ and \eqref{eq:disjoint} that $\Link_j(\sigma)$ is the following disjoint union
\begin{equation}\label{eq:union}
\Link_j(\sigma)=\dot{\bigcup}_{\tau \subseteq \sigma, |\tau|=j}\Link_{\tau}(\sigma).
\end{equation}

We have that if $\tau \subset \sigma$, \begin{equation} \label{eq:linktau} \Link(\tau)=(\sigma\setminus \tau) \dot{\cup} (\dot{\bigcup}_{\tau \subset \pi \subseteq \sigma} \Link_{\pi}(\sigma)).\end{equation}
It is immediate to check that all the elements in the right-hand side of the equation  \eqref{eq:linktau} lie in $\Link(\tau)$. It only remains to check that every $v\in \Link(\tau)$ either lies in $(\sigma\setminus \tau)$ or in $\Link_{\pi}(\sigma)$, for some $\tau\subseteq \pi\subseteq \sigma$.  If $v\notin \sigma$,   let $\pi=\Star(v)\cap \sigma$ and notice that since $v\in \Link(\tau)$, $\tau\subset \pi$. Thus $v\in \Link_\pi (\sigma).$ The fact that the unions are disjoint follows from \eqref{eq:disjoint}. 

Now let $f(i,j)$ be the number of $j$-cliques in an $i$-clique, where $j \leq i$.

We claim that $\Link_{\tau}(\sigma)$ depends only on $i$ and $j$ and the values of $|\Link(k)|$.
We prove the claim by induction on $i-j$, i.e.  for any $\tau \subset \sigma$ of size  $j$ the size of  $\Link_{\tau}(\sigma)$ depends only on $i$ and $j$ and the values of $|\Link(k)|$, and not on $\sigma$ and $\tau$. For $i-j=1$, by  \eqref{eq:linktau}, we have $\Link_{\tau}(\sigma)=(\Link(\tau)\setminus \sigma)\setminus \Link(\sigma) $, where $|\tau|=i-1$. Thus $\deg_{\tau}(\sigma)=|\Link_{\tau}(\sigma)|=(|\Link(i-1)|-1)- |\Link(i)|$ depends only on $i$ and $j=i-1$.

In general, by  \eqref{eq:linktau}, we have $\Link_{\tau}(\sigma)=\Link(\tau) \setminus ((\sigma\setminus \tau)  \cup(\dot{\bigcup}_{\tau \subset \pi \subseteq \sigma} \Link_{\pi}(\sigma))), $ and since the unions in \eqref{eq:linktau} are disjoint,   
 \begin{equation}\label{eq:sum}
 \deg_{\tau}(\sigma)=|\Link(\tau)|-|\sigma\setminus \tau|-\sum_{\tau \subset \pi \subseteq\sigma}|\Link_\pi(\sigma)|.
 \end{equation}
So we can assume that $i-j> 1$ and that  $\deg_{\tau'}(\sigma')$, $\tau'\subseteq \sigma'$ only depends of $|\tau'|$ and $|\sigma'|$ and the values of $|\Link(k)|$ if $|\sigma'|-|\tau'|<i-j$. 
So in \eqref{eq:sum}, $|\Link_{\pi}(\sigma)|$ depends only on $|\pi|>j$ and $i$ and the values of $|\Link(k)|$, the number of  $\pi\subset \sigma$ of a given size $j'$ is $f(i,j')$, which only depends on $i$ and $j'$. We obtain that  $\deg_{\tau}(\sigma)=|\Link_{\tau}(\sigma)|$ depends on $j$, $i$ and the values of $|\Link(k)|$ only. 
This completes the proof of the claim.

We complete the proof of the Lemma by observing that, by \eqref{eq:union}, $\deg_j(\sigma)=f(i,j)\deg_{\tau}(\sigma)$. Thus $\deg_j(\sigma)$ only depends on $j$, $|\sigma|$ and the values $|\Link(k)|$.  \end{proof}

\begin{Def}
Let $G$ be right-angled Coxeter group based on $\ga$ and  $\mathcal{D}$ be the associated automaton described in  Propositon \ref{prop:automaton}. We say that a state $\sigma\in \mathcal{A}$ of $\mathcal{D}$ is an $i$-state if $|\sigma|=i$. In particular there is only one $0$-state, which corresponds to the starting state.
\end{Def}
\begin{Lem}\label{lem:transj}
Fix  $\sigma \in \mathcal{A}$  and consider the corresponding state in $\mathcal{D}$. The number of transitions (edges) from $\sigma$ to $(j+1)$-states is $\deg_j(\sigma)$. 
\end{Lem}

\begin{proof}
By Proposition \ref{prop:automaton} (1), the transition function is given by $\mu(\sigma,v)=\{v\}\cup (\Star(v)\cap\sigma)$ for $v\notin \sigma$. Then $\mu(\sigma,v)$ is a $(j+1)$-state if and only if $\Star(v)\cap\sigma$ has size $j$, which is equivalent to $v\in \Link_j(\sigma)$. Then there are exactly $\deg_j(\sigma)$ transition edges from $\sigma$ to $j$-states. 
\end{proof}

The following is an analog of Theorem 14, \cite{GS2009}.
\begin{Thm} \label{Thm:GGC}
Let $(G,S)$ be a right-angled Coxeter system with link-regular defining graph $\ga$. The geodesic growth of $G$ only depends on the $f$-polynomial of $\ga$ and the values $|\Link(\sigma)|,$  $\sigma \in \mathcal{A}(G,S)$.
\end{Thm} 

\begin{proof}
Since $\ga$ is link-regular, and in view of Lemma \ref{lem:degj}, we can write $\deg_j(i)$ for $\deg_j(\sigma)$ and $|\Link(i)|$ for $|\Link(\sigma)|$, $|\sigma|=i$.

Let  $\mathcal{D}$ be the deterministic automaton recognizing the language of geodesics $\mathcal{G}$ of $G$. Let $d=\max\{|\sigma|:\sigma\in\mathcal{A}\}$. Then for each $i \geq 0,$ let $B_i(m)$, $m \geq 0$ denote the number of words of length $m$, with respect to $S$, that are accepted by an $i$-state (the starting state is a $0$-state). The geodesic growth series can be written as the double sum
$$\mathcal{G}(z)=1 + \sum_{m=1}^{\infty}\sum_{i=1}^{d} B_i(m)z^m.$$

The theorem follows if we show that  $B_i(m)$ depends only on the $f$-polynomial and the values of $|\Link(\sigma)|$ for $\sigma \in \mathcal{A}$.  We will argue by induction on $m$.
Notice that for $m< i$, $B_i(m)$ is zero and $B_i(i)/i!$ is the number of $i$-cliques, which is the information contained in the $f$-polynomial. In particular, this shows that for $m\leq 1$, $B_i(m)$, $i=0,\dots, d$ only depend on the $f$-polynomial (and trivially on the values $|\Link(i)|$).

Suppose that $m\geq 2$. We form a recursion for $B_i(m)$ as follows. Any word $w$ of length $m$ that is accepted by an $i$-state $\sigma$  is obtained by multiplying a word $w'$ of length $m-1$ by some generator in $V\Gamma$. Given $w'$ of length $m-1$ accepted by the $j$-state $\tau$, let $\beta_i(\tau)$ be the number of transitions from $\tau$ to an $i$-state. First notice that $i\leq j+1$. 

By Lemma \ref{lem:transj} $\beta_i(\tau)=\deg_{(i-1)}(\tau)$. By Lemma \ref{lem:degj}, $\deg_{(i-1)}(\tau)$ only depends on  $j=|\tau|$ and the values of $\Link(k)$, $k=0,\dots, d$. We write $\beta_{i,j}$ for $\beta_i(\tau)$. 


 One then has the recurrence
$$B_i(m)=\beta_{i,0}B_0(m-1) + \beta_{i,1}B_1(m-1)+ \dots + \beta_{i,i-1}B_{i-1}(m-1),$$ for $m \geq 0$. 

By the induction hypothesis, $B_i(m-1)$ only depend on the $f$-polynomial and $|\Link(i)|$, $0\leq i \leq d$, and thus the statement also follows for $B_i(m)$.
\end{proof}

\begin{Cor} \label{cor:sameGGC}
Let $\ga_1$ and $\ga_2$ be two graphs with the same $f$-polynomial, such that $|\Link(\sigma)|=|\Link(\sigma')|$ for all $\sigma \in \mathcal{A}(\ga_1)$ and $\sigma' \in \mathcal{A}(\ga_2)$ of the same size. Then the corresponding right-angled Coxeter groups $G_1$ and $G_2$ based on $\ga_1$ and $\ga_2$ have the same geodesic growth.
\end{Cor} 

To show the right-angled Artin version of the previous Corollary we need the following definiton.
\begin{Def}
The \textit{double} of a graph $\ga$ has two vertices $v^{(1)}$ and $v^{(2)}$ for each vertex $v$ of $\ga$, and if $v$ and $u$ are joined by an edge in $\ga$, then each of $v^{(1)}$ and $v^{(2)}$ is joined to each of $u^{(1)}$ and $u^{(2)}$.
\end{Def}
See Figure \ref{fig:double} for an example.
\begin{figure}[ht]

\begin{center}
\includegraphics{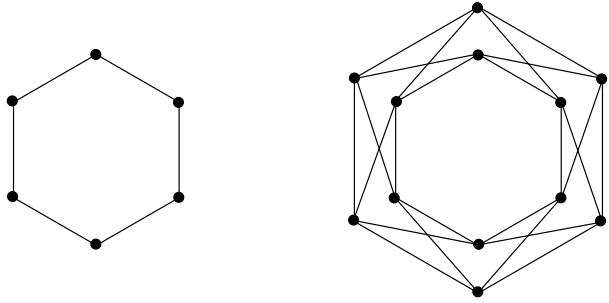}

\caption{The hexagon and the double of the hexagon.}\label{fig:double}
\end{center}
\end{figure}

\begin{Cor} \label{cor:sameGGA}
Let $\ga_1$ and $\ga_2$ be two graphs with the same $f$-polynomial, such that $|\Link(\sigma)|=|\Link(\sigma')|$ for all $\sigma \in \mathcal{A}(\ga_1)$ and $\sigma' \in \mathcal{A}(\ga_2)$ of the same size. Then the corresponding right-angled Artin groups $G_1$ and $G_2$ based on $\ga_1$ and $\ga_2$ have the same geodesic growth.
\end{Cor}

\begin{proof}
Let $\ga_1^2$ and $\ga_2^2$ be the doubles of $\ga_1$ and $\ga_2$.  Let $f_{\ga_1}(t)$ and $f_{\ga_1^2}(t)$ be the $f$-polynomials of $\ga_1$ and its double, respectively. It is easy to check that each clique $\sigma\in \mathcal{A}(\ga_1)$ of size $i$  produces $2^i$ cliques of size $i$ in $\mathcal{A}(\ga_1^2)$, and these are the only $i$-cliques in $\mathcal{A}(\ga_1^2)$. Thus $f_{\ga_1^2}(t)=f_{\ga_1}(2t)$, and therefore $\ga_1^2$ and $\ga_2^2$ have the same $f$-polynomial. It is also easy to see that the link of an $i$-clique of $\mathcal{A}(\ga_1^2)$ has twice as many vertices as that of the link of an $i$-clique in $\mathcal{A}({\ga_1})$. This shows that $\ga_1^2$ and $\ga_2^2$ are also link-regular, with same size links of identical size cliques.

 Lemma 2 in \cite{DromsServatius} states that the Cayley graph of the right-angled Artin group based on a graph $\ga$ is isomorphic as undirected graph to the Cayley graph of the right-angled Coxeter group based on $\ga^2$.  Then, by Corollary \ref{cor:sameGGC}, the racg's based on $\ga_1^2$ and $\ga_2^2$ have identical geodesic growth, and by (\cite{DromsServatius}, Lemma 2), so do the raag's based on $\ga_1$ and $\ga_2$.

\end{proof}

\begin{Rem} The smallest examples of racg's (or raag's) with the same geodesic growth provided by Theorem \ref{Thm:GGC} are those based on a cycle of length $8$, on one hand, and based on the union of two cycles of length $4$, on the other hand (see Figure \ref{fig:squares}). These two groups are non-isomorphic, see for example \cite[Theorem 5.2]{Bahls}. For more examples of such pairs of graphs see Examples 15 and 16 in \cite{GS2009}.
\end{Rem} 

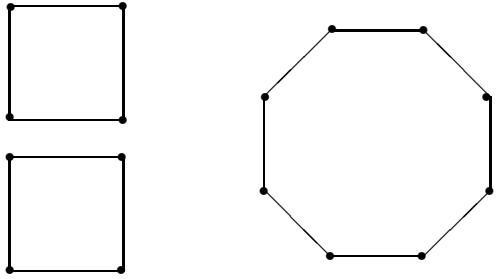
\begin{figure}[ht]
\begin{center}
\setlength{\unitlength}{3947sp}%
\begingroup\makeatletter\ifx\SetFigFont\undefined%
\gdef\SetFigFont#1#2#3#4#5{%
  \reset@font\fontsize{#1}{#2pt}%
  \fontfamily{#3}\fontseries{#4}\fontshape{#5}%
  \selectfont}%
\fi\endgroup%
\begin{picture}(3086,1735)(1118,-3480)
\thinlines
{\color[rgb]{0,0,0}\put(1155,-3443){\framebox(710,709){}}
}%
{\color[rgb]{0,0,0}\put(3741,-3354){\circle*{58}}
}%
{\color[rgb]{0,0,0}\put(4167,-2941){\circle*{58}}
}%
{\color[rgb]{0,0,0}\put(4149,-2355){\circle*{58}}
}%
{\color[rgb]{0,0,0}\put(3753,-1936){\circle*{58}}
}%
{\color[rgb]{0,0,0}\put(3180,-1924){\circle*{58}}
}%
{\color[rgb]{0,0,0}\put(2760,-2355){\circle*{58}}
}%
{\color[rgb]{0,0,0}\put(2749,-2944){\circle*{58}}
}%
{\color[rgb]{0,0,0}\put(2753,-2350){\line( 0,-1){588}}
\put(2752,-2938){\line( 1,-1){415.500}}
\put(3167,-3354){\line( 1, 0){588}}
\put(3755,-3356){\line( 1, 1){416}}
\put(4172,-2941){\line( 0, 1){589}}
\put(4173,-2352){\line(-1, 1){415.500}}
\put(3758,-1936){\line(-1, 0){588}}
\put(3170,-1935){\line(-1,-1){416}}
}%
{\color[rgb]{0,0,0}\put(1853,-3443){\circle*{58}}
}%
{\color[rgb]{0,0,0}\put(1155,-3443){\circle*{58}}
}%
{\color[rgb]{0,0,0}\put(1858,-2728){\circle*{58}}
}%
{\color[rgb]{0,0,0}\put(1156,-2734){\circle*{58}}
}%
{\color[rgb]{0,0,0}\put(1865,-2498){\circle*{58}}
}%
{\color[rgb]{0,0,0}\put(1865,-1782){\circle*{58}}
}%
{\color[rgb]{0,0,0}\put(1155,-2480){\circle*{58}}
}%
{\color[rgb]{0,0,0}\put(1162,-1787){\circle*{58}}
}%
{\color[rgb]{0,0,0}\put(1155,-2498){\framebox(710,710){}}
}%
{\color[rgb]{0,0,0}\put(3168,-3354){\circle*{58}}
}%
\end{picture}%

\caption{The left graph consists of two cycles of length 4. The right graph consists of a cycle of length 8.}\label{fig:squares}
\end{center}
\end{figure}

\begin{Rem} One cannot drop the condition that the graphs have the same $f$-polynomial. For example, for the racg based on a cycle of length $6$ the geodesic growth series is $ \dfrac{1+z+2z^2}{1-5z+2z^2}$ (this can been seen by employing the formula from Theorem \ref{formula}), and the racg based on the union of two cycles of length $3$ has geodesic growth series $\dfrac{1+3z+6z^2+6z^3}{1-3z-6z^2-6z^3}$ (this can be seen using the formula for the geodesic growth of free products). Then these two groups are link-regular, have the same number of generators, but have different geodesic growth.
\end{Rem}

We want to finish this section with two natural questions. The 
first one is: to what extent is our theorem optimal, i.e. are there other 
pairs of groups with the same geodesic growth? For example, if we 
restrict our attention to groups based on trees, computer experiments 
indicate that racg's based on non-isomorphic trees have different geodesic 
growth. Is the geodesic growth a complete invariant for groups based on 
trees?

\section{Relation with other types of growth} \label{relation}

If $(G,S)$ is a right-angled Coxeter system, $\mathcal{A}$ can be regarded as a generating set of $G$ via $\sigma\mapsto \prod_{s\in \sigma} s \in G$. In this context, the set $\mathcal{A}$ is called the {\it automatic generating set} of $G$.

This term was used in  \cite{GS2009} for right-angled Coxeter groups, and its justification is the following.
In \cite{NibloReeves98}, Niblo and Reeves showed that a group acting cocompactly and properly discontinously on a CAT(0) space has
an induced automatic structure with respect to a generating set associated to the action. In the case of right-angled Coxeter groups, the automatic structure
coming from the action on the Davis complex is with respect to the generating set $\mathcal{A}$.

The set $\mathcal{A}$ coincides with  the set of nonempty simplices in $K$, the {\it flag completion} or {\it nerve} of $\Gamma$, that is, the abstract simplicial complex whose simplices are elements of $\mathcal{A}.$ Then, for $\sigma\in \mathcal{A}$ we can define its K-link as

$$\KLink(\sigma)=\{\tau\in K:\tau\cup \sigma \in K\} .$$

In \cite{GS2009}, the authors use a different definition of \textit{link-regularity}, which we will call K-\textit{link-regularity}.  They say that $K$ is K-\textit{link-regular} if the $K$-link of every $i$-simplex,  $0 \leq i \leq dim(K)$, has the same $f$-polynomial. We here show that these definitions are in fact equivalent.

Requiring that the K-link of every $i$-simplex,  $0 \leq i \leq dim(K)$, has the same $f$-polynomial clearly implies that the link of every $i$-clique depends only on $i$. This is because for $\sigma\in\mathcal{A}$, $|\Link(\sigma)|$ is the $t$-coefficient in the $f$-polynomial of $\KLink(\sigma)$. This shows one implication of the equivalence. The following result shows the other implication, i.e. that Definition \ref{link-regular} imposes the equality of the $f$-polynomials of the K-links.

\begin{Lem}
Let $\Gamma$ be a link-regular graph. Then the K-link of  every $i$-simplex,  $0 \leq i \leq dim(K)$, has the same $f$-polynomial.
\end{Lem}

\begin{proof}
An $i$-simplex $\sigma\in K$ will be also regarded as an $(i+1)$-clique of $\ga$.

The standard result that the number of edges in a graph is half the sum of vertex degrees in that graph can be generalized to the following. For any graph $\Gamma$, the number of $(i+1)$-cliques in $\Gamma$ is given by $\frac{1}{i+1}\sum_{\tau \subseteq K, |\tau|=i}|\Link(\tau)|.$ To see this, notice that for $i$-clique $\tau$ and $v\in \Link(\tau)$, $\tau \cup v$ represents an $(i+1)$-clique, and that when enumerating all pairs  $(\tau, v)$, with $|\tau|=i$ and $v \in \Link(\tau)$, each $(i+1)$-clique gets counted $i+1$ times, once for each vertex in it.

Now fix an i-simplex $\sigma$ in $K$. Its $f$-polynomial $f^{\sigma}(t)=1+f_0^{\sigma}t +f_1^{\sigma}t^2+ \dots $ is determined by the numbers $f_j^{\sigma}$ of $j$-simplices in $\KLink(\sigma)$, for all $0\leq j \leq d-i$. We will prove by induction on $j$ that $f_j^{\sigma}$ depends only on $i$ and $j$, and not on the simplex $\sigma$. For $j=0$ we have $f_0^{\sigma}=|\Link(\sigma)|=|\Link(i+1)|$. By the above $$f_j^{\sigma}=\frac{1}{j+1}\sum_{\tau \subseteq \Link(\sigma), |\tau|=j}|\Link(\tau)\cap \Link(\sigma)|$$ is the number of $(j+1)$-cliques in $\Link(\sigma)$, in other words, the number of $j$-simplices in $\KLink(\sigma)$. Since $\Link(\tau)\cap \Link(\sigma)= \Link(\sigma \cup \tau)$, we get $$f_j^{\sigma}=\frac{1}{j+1}\sum_{\sigma \subset \pi \subseteq K, |\pi|=i+j}|\Link(\pi)|=\frac{1}{j+1}\sum_{\sigma \subset \pi \subseteq K, |\pi|=i+j}|\Link(i+j)|.$$ 
However, the number of $(i+j)$-cliques containing $\sigma$ is equal to $f_{(j-1)}^{\sigma}$, and since by induction this only depends on $i$ and $j$, we obtain that $f_{j}^{\sigma}$ is a function of only $i$ and $j$ as well. Therefore the K-link of  every $i$-simplex,  $0 \leq i \leq \text{dim}(K)$, has the same $f$-polynomial.
\end{proof}

It is well known that the spherical growth of raag's and racg's with respect to the standard generating set depends entirely on the $f$-polynomial of the generating graph (\cite{chiswell}). However, as pointed out in \cite{GS2009} and \cite{JJJ}, this is not the case for  the geodesic growth of racg's or the spherical growth with respect to the automatic generating set. That is, there are racg's with the same standard growth, i.e. same $f$-polynomial, but different automatic or geodesic growth.

Hence, if two graphs $\ga_1$ and $\ga_2$ are link-regular, have the same $f$-polynomial and $|\Link(\sigma)|=|\Link(\sigma')|$ for all $\sigma\in \mathcal{A}(\ga_1)$ and $\sigma'\in\mathcal{A}(\ga_2)$ with $|\sigma|=|\sigma'|$, then the respective racg's $G(\ga_1)$ and $G(\ga_2)$ have
\begin{enumerate}
\item[(i)]  the same geodesic growth with respect to the standard generating set,
\item[(ii)] the same spherical growth with respect to the automatic generating set,
\item[(iii)] the same spherical growth with respect to the standard generating set.
\end{enumerate}

It is not clear what the relation between the different types of 
growth considered in this section is in general. For example, it is not clear if there exist racg's with the same geodesic growth series but different spherical 
growth series. Also, it is not clear if, for two groups, equal 
geodesic growth implies equal spherical growth with respect to the 
automatic generating or viceversa.
\section{Geodesic growth formula for racg's based on regular triangle-free graphs}\label{sec:formula}

Here we provide the geodesic growth formula for racg's determined by regular triangle-free graphs.

\begin{Thm}\label{formula}
Let $\ga$ be an $l$-regular triangle-free graph with $n$ vertices, $n\geq 4$. The geodesic growth series for the right-angled Coxeter group based on $\ga$ is
\begin{equation}\label{eq:THEformula}
\mathcal{G}(z)=\dfrac{1-(l-3)z+2z^2}{1+(-n-l+3)z+(-2n+2+nl)z^2}
\end{equation}

\begin{proof}
The proof is a consequence of Lemma \ref{lem:sum} and the formulas \eqref{eq:sum3}, \eqref{eq:Eu} and \eqref{eq:Ee}, which follow.\end{proof}
\end{Thm}

Let  $\Gamma$ be a regular, triangle-free finite graph with $n$ vertices and $l=|\Link(v)|$, for all $v\in \Gamma V$. Let $\widehat{\ga}$ be its corresponding oriented graph, that is, $\widehat{\ga}$ has the same vertex set as $\ga$ and if $u,v$ were connected by an edge in $\ga$, they are connected by two edges in $\widehat{\ga}$, one from $u$ to $v$ and other from $v$ to $u$. 

Let $\mathcal{G}$ denote the language of geodesics of the associated right-angled Coxeter group.
Denote by $E_u$, $E_{uv}$ and $E_{uvt}$ the sets of geodesic words ending in $v$, $uv$ and $uvt$, respectively, where $u,v,t \in V\ga$. We have 
\begin{align}
\mathcal{G}(z)& =1+\sum_{u\in V\ga} E_u(z)\label{eq:sum1},\\
\mathcal{G}(z)& =1+nz+\sum_{u,v\in V\ga} E_{uv}(z)\label{eq:sum2},\\
\mathcal{G}(z)& =1+nz+n(n-1)z^2+\sum_{u,v,t\in V\ga} E_{uvt}(z),\label{eq:sum3}
\end{align}
 where $E(z)$ denotes the characteristic formal power series of a set $E$.
From \eqref{eq:sum1} we obtain
\begin{equation}\label{eq:Eu}
 \sum_{v\in V\ga} E_v(z)= \mathcal{G}(z)-1.
\end{equation}

\begin{Lem}\label{lem:uv2}

\begin{itemize}
\item[\rm{(1)}] If $u\notin \Star(v)$ then $E_{uv}=E_u\cdot v.$ In particular, $E_{uv}(z)= E_v(z)\cdot z$.
If $u=v$ then $E_{uv}$ is empty and $E_{uv}(z)=0$.

\item[\rm{(2)}] If $\{u,v\}\not\subset \Star(t)$ then $E_{uvt}=E_{uv}\cdot t$. In particular, $E_{uvt}(z)= E_{uv}(z)\cdot z$.

\item[\rm{(3)}] If $\{u,v\}\subset \Star(t)$ then 
\begin{enumerate}
\item[\normalfont{(i)}] If $u\neq t\neq v\neq u$, then $E_{uvt}(z)=E_{utv}(z)=E_{ut}(z)\cdot z=E_e(z)\cdot z$, where $e$ is the edge from $u$ to $t$.
\item[\normalfont{(ii)}] If $u=t$ or $v=t$ or $u=v$ then $E_{uvt}(z)=0.$
\end{enumerate}
\end{itemize}
\end{Lem}
\begin{proof}
(1) The second fact is obvious. For the first one, let $w$ be a word ending in $u$. If $(w,v)$ is not geodesic then $w$ can be factorized as $(w_1,v ,w_2)$, with all the generators involved in $w_2$ belonging to $\Star(v)$. Since $w_2$ ends in $u \notin \Star(v)$, $(w,v)$ is geodesic.

(2) The proof is the same as the one for (1).

(3) The statement (ii) is obvious. We prove (i). Let $w\in E_{uvt}$. Then $w=(w_1,u,v,t)$. Since $t$ and $v$ commute, $(w_{1},u,t,v)$ is geodesic and this defines a bijection between $E_{uvt}$ and $E_{utv}$. Now, since $\Gamma$ is triangle free, if $t\neq v\neq u$ and $u,v\in \Star(t)$ then $v\notin \Star(u)$ and hence $E_{utv}=E_{ut} \cdot v$, by (2). Finally $u$ and $t$ are endpoints of an edge.
\end{proof}

For each edge $e\in E\widehat{\ga}$, we define $E_e(z)$ to be the characteristic function of $E_{uv}$, where $u$ is the initial vertex of $e$ and $v$ is the terminal vertex of $e$.
From \eqref{eq:sum2},  \eqref{eq:Eu} and the Lemma \ref{lem:uv2}(1) we obtain
\begin{align*}
\mathcal{G}(z) 
&=1+ nz+ \sum_{v\in V\ga, u\notin \Star(v)} E_{u}(z) \cdot z +\sum_{e\in E\widehat{\ga}} E_{e}(z)\\
&= 1+nz+(n-l-1) (\mathcal{G}(z)-1) z+ \sum_{e\in E\widehat{\ga}} E_{e}(z).
\end{align*}
Then finally,
\begin{equation}\label{eq:Ee}
\sum_{e\in E\widehat{\ga}} E_{e}(z)=
\mathcal{G}(z)(1-(n-l-1)z) -1-(l+1)z.
\end{equation}
Recall that $|V\ga|=n$ and $|\Link(v)|=l$.

\begin{Lem}\label{lem:sum}
$$\sum_{u,v,t\in V\ga} E_{uvt}(z)=(n+l-3) \sum_{e\in E\widehat{\ga}} E_e(z)z+( -l^2+2l+n^2-2n-nl+1)\sum_{u\in V\ga}E_u(z)z^2.$$
\end{Lem}
\begin{proof}
We will split the sum  $\sum_{u,v,t\in V\ga} E_{uvt}(z)$ in different cases depending on  $t$.

{\bf Case 1:} $\{u,v\} \subset \Star(t)$.

If $\{u,v\}\subset\Star(t)$, then by Lemma \ref{lem:uv2}(3) we are only interested in the cases $t\neq u\neq v\neq t$. Then $E_{uvt}(z)=E_e(z) z$, where $e$ is the edge from $u$ to $t$. 

\begin{itemize}
\item Fixing $t$ and $u$, there are $l-1$ posibilities for $v$, that is,  $E_e(z) z$, where $e$ is the edge form $u$ to $t$, appears $l-1$ times in this case. Hence the contribution of these triples is $$(l-1)\sum_{e\in E\widehat{\ga}} E_e(z)z.$$
\end{itemize}

{\bf Case 2:} $\{u,v\} \not\subset \Star(t)$.

{\bf Subcase 2.1:} Suppose that $u$ and $v$ are at distance 1. 

Then $E_{uvt}(z)=E_e(z)\cdot z$ by Lemma \ref{lem:uv2} (2), where $e$ is the edge from $u$ to $v$.
\begin{itemize}
\item   For each edge from $u$ to $v$ there are $n-2$ vertices different from $\{u,v\}$. Since $\ga$ is triangle free, if $t$ is any of these $n-2$ vertices, $\{u,v\}\not\subset \Star(t)$. The contribution of these triples to the sum is $$(n-2) \sum_{e\in E\widehat{\ga}} E_e(z)z.$$
\end{itemize}

 {\bf Subcase 2.2:} Suppose that $u$ and $v$ are at distance greater than 1. 

Since $\{u,v\}\not\subset\Star(t)$ and $t\neq v$, then $E_{uvt}(z)=E_{uv}(z)z=E_u(z)z^2$, by Lemma \ref{lem:uv2} (1), (2). Notice that $v$ can be any vertex except for $t$ and the vertices between $u,v$ that are at distance more than $1$ from $v$.
\begin{itemize}
\item   Fixing $t$, there are $n-1$  possibilities for $v$. There are $n-l-1$ vertices at distance more that 1 from $v$. We have to substract the cases when $u$ and $v$ are at distance 2 and in the link of $t$. This quantity is $l(l-1)$. For each of these $t$, there are $(n-1)(n-l-1)-l(l-1)$ possibilities for $u$ and $v$, i.e., for a given $u$, there are $(n-1)(n-1-l)-l(l-1)$ possibilities for $v$ and $t$ in this case. The contribution of these triples to the sum is $$[(n-1)(n-1-l)-l(l-1)]\sum_{u\in V\ga}E_u(z)z^2.$$
\end{itemize}

It is easy to check that the total sum is exactly the one claimed in the lemma.
\end{proof}
\begin{Rem}
It is possible to use the ideas of Theorem \ref{formula} to obtain a formula for the geodesic growth series of  any ragc based on a link-regular graph. However, the combinatorics involved get more complicated. 

This approach can also be used to give an alternative proof of Theorem \ref{Thm:GGC}. 
If $d=\max\{ |\sigma|: \sigma \in \mathcal{A}\}$, one can use an inductive argument on $d$ to show that $\mathcal{G}(z)$ is rational and only depend on the $f$-polynomial of $\ga$ and the values of $|\Link(k)|$, $k=1,\dots, d$. 

One argues as follows: first one uses link-regularity to generalize Lemma \ref{lem:sum}. That is, the sum of the characteristic series of words ending in suffixes of length $d+1$ is a linear combination of the sums of the characteristic series of words ending in $j$ letters spanning a $j$-clique, $1\leq j \leq d$.

Then  one uses the induction hypothesis to argue that there exist $d+1$ formulas as \eqref{eq:sum1}, \eqref{eq:sum2},\eqref{eq:sum3} that only depend   on the $f$-polynomial of $\ga$ and the values of $|\Link(k)|$, $k=1,\dots, d$.

Theorem \ref{formula} follows this inductive strategy for the case $d=2$.
\end{Rem}

\section{Geodesic growth of even Coxeter groups} \label{evenCoxeter}

The objective of this section is to use  Grigorchuk and Nagnibeda's technique \cite{GN} to describe the  geodesic growth of triangle-free even Coxeter groups, and show that two star-regular, triangle-free, even Coxeter groups with the same number of generators have the same geodesic growth.

Before plunging into the geodesic growth we recall a few facts about Coxeter systems. We start with a standard characterization of Coxeter systems, a proof of which can be found, for example, in \cite[Theorem 1.5.1]{BjornerBrenti}.
\begin{Thm}[The Deletion Condition]
Let $G$ be a group generated by a set of involutions $S\subseteq G.$ Then $(G,S)$ is a Coxeter system if and only if the following holds: for every word $\word\equiv(x_1,\dots,x_n)$ in $S$ that is not geodesic, there exist indices $i<j$ such that $(x_1, \dots,\hat{x_{i}},\dots,\hat{x_{j}},\dots, x_n)$, the word obtained by deleting $x_i$ and $x_j$,  represents the same element of $G$ as $\word$.
\end{Thm}

A nice property of even Coxeter groups is that the abelianization consists of $|S|$ copies of $\Z/2\Z$.
\begin{Lem}\label{lem:abelianized}
Let $(G,S)$ be a Coxeter system. Then there exists a homomorphism from $G$ to $\Z/2\Z$.
Moreover, if $(G,S)$ is an even Coxeter system, then $G_{\text{ab}}\cong (\Z/2\Z)^{|S|}.$  
\end{Lem}
\begin{proof}
When we abelianize $G$, the relation $(sr)^m$ is either a consequence of $sr=rs$ (when $m$ is  even) or implies that $s=r$  (when $m$ is odd). Hence, if $G$ is non-trivial, the abelianization is a non-trivial  product of $\Z /2\Z$'s. If $(G,S)$ is even, then $G_{\text{ab}}\cong (\Z/2\Z)^{|S|}.$ 
\end{proof}
The following observation will be useful. 
\begin{Cor}\label{cor:forbidden}
Let $(G,S)$ be an even Coxeter system and $\word \equiv (x_1,\dots, x_n)$ be a geodesic.
Suppose that there exist $s,t\in S$ such that $(s,\word) $ and $(\word, t)$ are geodesics, but $(s, \word, t)$ is not geodesic. Then $(s,\word, t)$ represents the same word as $\word$ and $s=t.$
\end{Cor}
\begin{proof}
Since $(s, w, t)$ is not geodesic, the Deletion Condition implies that we can suppress two letters of $(s,\word, t).$ As $(s, \word)$ and $(\word, t)$ are geodesics, the only possibility is to suppress $s$ and $t$.  Lemma \ref{lem:abelianized} implies $s=t$.
\end{proof}

From now on we will use the terminology from Grigorchuk and Nagnibeda's paper \cite{GN}. Let $\mathcal{G}$ denote the language of geodesics of a Coxeter group over the set $S$. To compute the growth of the language of geodesics, we will describe $\mathcal{G}$ as a language determined by a set of {\it forbidden words} $U$. That is, there is a set $U\subseteq S^*$, such that any word on $S$ is in $\mathcal{L}$ exactly when it does not contain a subword of $U$. 
According to \cite{GN}, the set $U$ must satisfy the two natural conditions: 
\begin{enumerate}
\item[(U1)]  $S\cap U=\emptyset$; 
\item[(U2)] any forbidden word contains no proper forbidden subword.
\end{enumerate}

For $s\in S$, let $C(s)=\{g\in G: gs=sg\}$, the centralizer of $s$. 

\begin{Lem}\label{lem:langforbwords0}
Let $(G,S)$ be an even Coxeter system. For $s\in S$ let 
$$U(s)=\{(s,\word,s) : (\word,s) \text{ is a geodesic word such that} \ \pi(w) \in C(s) \}.$$ 
Then $\mathcal{G}$, the language of geodesics of $G$ over $S$, is determined by the set of forbidden words $U=U(G,S)\coloneqq \cup_{s\in S} U(s)$. Moreover, $U$  satisfies {\normalfont{(U1)}} and {\normalfont{(U2)}}.
\end{Lem}
\begin{proof}
The set $U$ satisfies (U1) since each word in $U$ has length at least 2. It also satisfies (U2) since words of $U$ are non-geodesics, but every proper subword is geodesic. In particular, any word $\word \in S^*$ containing a subword of $U$ is non-geodesic.

It only remains to show that if $\word \in S^*$ does not contain any subword of $U$, then $\word$ is geodesic. Suppose that this is not true, that is, there exists $\word \in S^*$ such that $\word$ does not contain any subword of $U$ and $\word$ is not geodesic. We can assume further that $\word$ is the shortest with this property, that is, any proper subword is geodesic. Since every element of $S$ is non-trivial in $G$, we can assume that $\word$ has length greater than 2.  We can express $\word$ as  $(s,\word_1,t)$ where $s,t\in S$ and $\word_1\in S^*.$ Then $(s,\word_1)$ and $(\word_1,t)$ are geodesics. By Corollary \ref{cor:forbidden}, $s=t$ and $(s, \word_1, s)$ represents the same word as $\word_1$. Then, by minimality, $(\word_1)$ is a geodesic word representing some element of $C(s)$. 
\end{proof}

The previous lemma tells us that, in order to understand the forbidden words, we need to understand the centralizers of generators. This was studied by Brink \cite{Brink} in terms of root systems.

\begin{Not}
Let $(G,S)$ be an even Coxeter system with relation matrix $(m_{s,t}).$ For $s,t\in S$, $m_{s,t}\neq 0$, we denote by $a_{t,s}$ the word $(t,s,t,s, \dots, t)$ of length $m_{s,t}-1$ (i.e. $(ts)^{m_{s,t}/2-1}t$), and  we let $A(s)$ be the set $\{a_{t,s} :  t \in \Link(s)\}$. \end{Not}

\begin{Thm}\label{thm:centralizers}
Suppose $(G,S)$ is an even Coxeter system. Then

\begin{enumerate}
\item[\normalfont{(i)}] $C(s)$ is a Coxeter group generated by $\{s\}\cup A(s).$
\item[\normalfont{(ii)}] If a word on $\{s\}\cup A(s)$ is geodesic, then it is also geodesic as a word on $S$.
\item[\normalfont{(iii)}] Let $A(s)=\{a_1,\dots,a_n\}.$ If $\Star(s)$ spans a tree $\ga(W,S)$, then $C(s)\cong \gp{s,a_1,\dots, a_n}{s^{2},a_i^{2}, [s,a_i]\, (i=1,\dots,n)}.$

\end{enumerate}
\end{Thm}

\begin{proof}
(i) is a restatement of  Brink's result \cite{Brink} in terms of generators, and it can be found in \cite[Theorem 2.6]{Bahls}.

(ii) is the content of the unpublished  \cite[Theorem 4.10]{BahlsMihalik} and we prove it in the appendix \ref{s:centralizers} for completeness.

(iii) follows from Brink's description of the centralizers.
\end{proof}

In view of the Theorem \ref{thm:centralizers} and Lemma \ref{lem:langforbwords0} we can give a more accurate description of the set of forbidden words

\begin{Lem}\label{lem:langforbwords}
Let $(G,S)$ be an even Coxeter system. For $s\in S$ let 
$$U(s)=\{(s,\word,s) : \word \text{ is a geodesic word in  } A(s) \}.$$ 
Then $\mathcal{G}$, the language of geodesics of $G$ over $S$, is determined by the set of forbidden words $U=U(G,S)\coloneqq \cup_{s\in S} U(s)$.\end{Lem}

In order to apply the technique of \cite{GN} we have to understand the (rigid) chains of forbidden words. A  word is called a  {\it chain} if, for some positive integer $m$, it admits a  ``staggered presentation of forbidden words'' $(u_{i})_{i=1}^m$ 
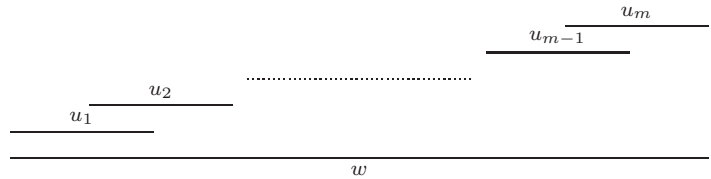
\begin{figure}[h]
\centerline{
\xymatrix@R=4pt{
& & & & & &  &\ar@{-}[rr]^{u_m}& &\\
& & & & & & \ar@{-}[rr]^{u_{m-1}} && &\\
& & & \ar@{..}[rrr]& & & & & &\\
& \ar@{-}[rr]^{u_2}  & &  & & & & & &\\
\ar@{-}[rr]^{u_1}  & & & & & & && &\\
\ar@{-}[rrrrrrrrr]_{w}& & & & & & && &}
}
\caption{Staggered presentation of forbidden words for $w$}
\label{fig:chainpres}
\end{figure}

\noindent where each line in Figure \ref{fig:chainpres} corresponds to an occurrence in $w$ of a forbidden subword $u_i\in U$, $i\in\{1,\dots,m\}$  such that  the intersection of two subwords $u_i$ and $u_j$ is non empty if and only if $|i-j|\leq 1.$

We introduce the formal definition of a chain in  Definition \ref{def:chain}. Before we need the following concept.
\begin{Def}
Let $w_1,w_2$ be two words over some alphabet. We define the {\it intersection} of $w_1$ and $w_2$, denoted $w_1\sqcap w_2$ to be the set of words $w_3$ that are a suffix of $w_1$ and a prefix of $w_2$. We say that a word $w$ is {\it the amalgamation of $w_1$ and $w_2$ over $w_3$}, if $w_3\in w_1 \sqcap w_2$ and $w\equiv(w_1',w_3,w_2')$ where $w_1'$ (resp. $w_2'$) is obtained from $w_1$ (resp. $w_2$) by deleting the suffix (resp. prefix) $w_3$. We will write $w\equiv w_1*_{w_3}w_2$.
\end{Def}

\begin{Def}[Chain]\label{def:chain}
A chain is a triple $(w, (u_{i})_{i=1}^m,(z_{i})_{i=1}^{m-1})$, where each $u_i$ is a forbidden word, $z_{i}$ is a non-empty word of $u_i\sqcap u_{i+1}$, $z_1\neq u_1$,  $|z_i|+|z_{i+1}|\leq|u_{i+1}|$, and $w=((\dots (u_1*_{z_1}u_2)*_{z_2}u_3)\dots *_{z_{m-1}}u_m$.
The length of the chain is the length of $w$.

A  {\it rigid chain of rank m} is defined by induction on $m$.
A  rigid chain of rank 1 is a forbidden word. A chain $(w, (u_{i})_{i=1}^2,(z_{i})_{i=1}^{1})$ of rank 2 is rigid if $w$ contains no forbidden subwords but  $u_1$ and $u_2$. Suppose that the rigid chains of rank $m-1$ have already been defined. Then a chain $(w, (u_{i})_{i=1}^m,(z_{i})_{i=1}^{m-1})$ is rigid of rank $m$ if $(w', (u_{i})_{i=1}^{m-1},(z_{i})_{i=1}^{m-2})$ is a rigid chain of rank $m-1$  and  the suffix of $w$ following $u_{m-2}$ contains  no forbidden subword except $u_m$.
\end{Def}

The relation between the complete growth of a language and the rigid chains is given in \cite[Theorem 2]{GN}. Using the ring homomorphism map $\Z[S^*]\to \Z$, $w\mapsto 1$ for all $w \in S^*$, we obtain the following key result. 

\begin{Thm}\label{thm:criterium}Let $\mathcal{L}\subseteq S^*$ be a language determined by a set of forbidden words $U$ satisfying (U1) and (U2). The growth function of $\mathcal{L}$ is completely determined by $|S|$ and the values of $|Q_n^{(m)}|,$ the number of rigid chains of rank $m$ and length $n$.\end{Thm}

We will use Theorem \ref{thm:criterium} to find two different even Coxeter system with the same geodesic growth.

For the rest of the section $(G,S)$ will be an even Coxeter systems and $U=U(W,S)$ the set of forbidden words described in Lemma \ref{lem:langforbwords}.
\begin{Lem}\label{lem:intersect}
Let $(G,S)$ be a triangle-free, even Coxeter system.  
Let $u_1,u_2\in U$,  such that $u_1\equiv(s,\omega_1,s)$ and $u_2\equiv(t,\omega_2,t)$, where $\omega_1$ is a geodesic  word in $A(s)$ and $\omega_2$ is a geodesic word in $A(t)$. Suppose that $w_2\in u_1\sqcap u_2$ and $w_2$ is a proper subword of $u_1$ and $u_2$. Then we can write $u_1\equiv(w_1,w_2)$, $u_2\equiv(w_2,w_3)$. Let $w=u_1*_{w_2}u_2$.  

\begin{enumerate}
\item[(i)] If $s\neq t$ then $w_2\equiv(t,a_{s,t})$ if and only if  $u_1$ and $u_2$ are the only forbidden subwords of $w$ and $w_2$ is non-empty.
\item[(ii)] If $s=t$ then $u_1\equiv(s,s)$ or $u_2\equiv(s,s)$  if and only if  $u_1$ and $u_2$ are the only forbidden subwords of $w$  and $w_2$ is non-empty.

In this case, if any of the two equivalent conditions hold, $w_2$ is forced to be equal to $(s)$.
\end{enumerate}
\end{Lem}
\begin{proof}
The diagram in Figure \ref{fig:wordw} depicts the relation between $w$ and the subwords involved in the Lemma.
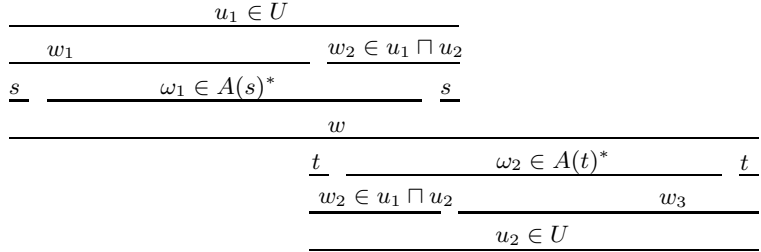
\begin{figure}[h]

\begin{center}
\setlength{\unitlength}{3947sp}%
\begingroup\makeatletter\ifx\SetFigFont\undefined%
\gdef\SetFigFont#1#2#3#4#5{%
  \reset@font\fontsize{#1}{#2pt}%
  \fontfamily{#3}\fontseries{#4}\fontshape{#5}%
  \selectfont}%
\fi\endgroup%
\begin{picture}(4732,1583)(1164,-3383)
\thinlines
{\color[rgb]{0,0,0}\put(3064,-3136){\line( 1, 0){823}}
}%
{\color[rgb]{0,0,0}\put(4004,-3136){\line( 1, 0){117}}
\put(4121,-3136){\line( 1, 0){118}}
\put(4239,-3136){\line( 1, 0){117}}
\put(4356,-3136){\line( 1, 0){118}}
\put(4474,-3136){\line( 1, 0){117}}
\put(4591,-3136){\line( 1, 0){118}}
\put(4709,-3136){\line( 1, 0){117}}
\put(4826,-3136){\line( 1, 0){118}}
\put(4944,-3136){\line( 1, 0){117}}
\put(5061,-3136){\line( 1, 0){118}}
\put(5179,-3136){\line( 1, 0){117}}
\put(5296,-3136){\line( 1, 0){118}}
\put(5414,-3136){\line( 1, 0){117}}
\put(5531,-3136){\line( 1, 0){118}}
\put(5649,-3136){\line( 1, 0){117}}
\put(5766,-3136){\line( 1, 0){118}}
}%
{\color[rgb]{0,0,0}\put(3064,-3371){\line( 1, 0){2820}}
}%
{\color[rgb]{0,0,0}\put(3299,-2901){\line( 1, 0){2350}}
}%
{\color[rgb]{0,0,0}\put(5766,-2901){\line( 1, 0){118}}
}%
{\color[rgb]{0,0,0}\put(3064,-2901){\line( 1, 0){118}}
}%
{\color[rgb]{0,0,0}\put(1185,-2667){\line( 1, 0){4699}}
}%
{\color[rgb]{0,0,0}\put(1420,-2432){\line( 1, 0){2349}}
}%
{\color[rgb]{0,0,0}\put(3887,-2432){\line( 1, 0){117}}
}%
{\color[rgb]{0,0,0}\put(1185,-2432){\line( 1, 0){117}}
}%
{\color[rgb]{0,0,0}\put(1185,-2197){\line( 1, 0){117}}
\put(1302,-2197){\line( 1, 0){118}}
\put(1420,-2197){\line( 1, 0){117}}
\put(1537,-2197){\line( 1, 0){118}}
\put(1655,-2197){\line( 1, 0){117}}
\put(1772,-2197){\line( 1, 0){118}}
\put(1890,-2197){\line( 1, 0){117}}
\put(2007,-2197){\line( 1, 0){117}}
\put(2124,-2197){\line( 1, 0){118}}
\put(2242,-2197){\line( 1, 0){117}}
\put(2359,-2197){\line( 1, 0){118}}
\put(2477,-2197){\line( 1, 0){117}}
\put(2594,-2197){\line( 1, 0){118}}
\put(2712,-2197){\line( 1, 0){117}}
\put(2829,-2197){\line( 1, 0){118}}
\put(2947,-2197){\line( 1, 0){117}}
}%
{\color[rgb]{0,0,0}\put(3182,-2197){\line( 1, 0){822}}
}%
{\color[rgb]{0,0,0}\put(1185,-1962){\line( 1, 0){2819}}
}%
\put(3182,-2138){\makebox(0,0)[lb]{\smash{{\SetFigFont{9}{10.8}{\rmdefault}{\mddefault}{\updefault}{\color[rgb]{0,0,0}$w_2\in u_1\sqcap u_2$}%
}}}}
\put(3182,-2629){\makebox(0,0)[lb]{\smash{{\SetFigFont{9}{10.8}{\rmdefault}{\mddefault}{\updefault}{\color[rgb]{0,0,0}$w$}%
}}}}
\put(2130,-2389){\makebox(0,0)[lb]{\smash{{\SetFigFont{9}{10.8}{\rmdefault}{\mddefault}{\updefault}{\color[rgb]{0,0,0}$\omega_1\in A(s)^*$}%
}}}}
\put(1416,-2148){\makebox(0,0)[lb]{\smash{{\SetFigFont{9}{10.8}{\rmdefault}{\mddefault}{\updefault}{\color[rgb]{0,0,0}$w_1$}%
}}}}
\put(5262,-3082){\makebox(0,0)[lb]{\smash{{\SetFigFont{9}{10.8}{\rmdefault}{\mddefault}{\updefault}{\color[rgb]{0,0,0}$w_3$}%
}}}}
\put(2471,-1919){\makebox(0,0)[lb]{\smash{{\SetFigFont{9}{10.8}{\rmdefault}{\mddefault}{\updefault}{\color[rgb]{0,0,0}$u_1\in U$}%
}}}}
\put(3081,-2859){\makebox(0,0)[lb]{\smash{{\SetFigFont{9}{10.8}{\rmdefault}{\mddefault}{\updefault}{\color[rgb]{0,0,0}$t$}%
}}}}
\put(5772,-2864){\makebox(0,0)[lb]{\smash{{\SetFigFont{9}{10.8}{\rmdefault}{\mddefault}{\updefault}{\color[rgb]{0,0,0}$t$}%
}}}}
\put(3887,-2389){\makebox(0,0)[lb]{\smash{{\SetFigFont{9}{10.8}{\rmdefault}{\mddefault}{\updefault}{\color[rgb]{0,0,0}$s$}%
}}}}
\put(1179,-2389){\makebox(0,0)[lb]{\smash{{\SetFigFont{9}{10.8}{\rmdefault}{\mddefault}{\updefault}{\color[rgb]{0,0,0}$s$}%
}}}}
\put(4239,-2853){\makebox(0,0)[lb]{\smash{{\SetFigFont{9}{10.8}{\rmdefault}{\mddefault}{\updefault}{\color[rgb]{0,0,0}$\omega_2\in A(t)^*$}%
}}}}
\put(3123,-3077){\makebox(0,0)[lb]{\smash{{\SetFigFont{9}{10.8}{\rmdefault}{\mddefault}{\updefault}{\color[rgb]{0,0,0}$w_2\in u_1\sqcap u_2$}%
}}}}
\put(4233,-3322){\makebox(0,0)[lb]{\smash{{\SetFigFont{9}{10.8}{\rmdefault}{\mddefault}{\updefault}{\color[rgb]{0,0,0}$u_2\in U$}%
}}}}
\end{picture}%

\caption{The word $w$ and the subwords $u_1$, $u_2$, $w_1$, $w_2$, $w_3$, $\omega_1$ and $\omega_2$.}\label{fig:wordw}
\end{center}

\end{figure}

(i) We first prove the {\it only if} part. Assume that $w_2\equiv(t,a_{s,t})$. By (U2), if there is a third forbidden subword $u_3$, it must  contain a suffix of $u_1$ longer than  $w_2$ (but shorter than $u_1$) and a prefix of $u_2$ longer than $w_2$ (but shorter than $u_2$). 
Suppose that $u_3=(r,\omega_3,r)$. As $u_3$ contains  $w_{2}$, we have that  $s,t\in \Star(r).$ Also, $t,r\in \Star(s)$ and $s,r\in \Star (t).$ Since $(G,S)$ is triangle-free and $t\neq s$, either $t=r$ or $s=r$.

We can assume $r=s$, the case $r=t$ is similar. Suppose that $u_1 \equiv (...,a_{p,s},a_{t,s},s)$ and $u_2 \equiv (t,a_{s,t},a_{q,t},...)$. Since such $u_3=(s,\dots,s)$ contains a prefix of $u_2$ longer than $w_2$, $q\in \Star(s)$. Since $\omega_2$ is a geodesic word over $A(t)$, Theorem \ref{thm:centralizers}(iii) implies that $q\neq s$, $q\neq t$, and $q \in \Link(t)$.  Then  $s, t \in \Star(q)$,  and we found a triangle, which is a contradiction of the hypothesis.

We now prove the {\it if} part. Since $w_2$ is non-empty, $u_1\equiv(\dots, a_{t,s},s)$, $u_2\equiv(t,a_{s,t},\dots)$, 
\begin{equation}\label{eq:intersec1}
w\equiv(s,\dots, \lefteqn{\overbrace{\phantom{t,s,\dots,s,t,\dots,t}}^{a_{t,s}}}t,s,\dots,
\underbrace{s,t,\dots,t,s,\dots, t,s}_{a_{s,t}},\dots, t)
\end{equation}
and, as suggested by \eqref{eq:intersec1},  there is a third forbidden subword in $w$ if $w_2$ is shorter than $(t, a_{s,t}) $. If $w_2$ is longer, then $u_2=(t, a_{s,t},a_{r,t},\dots,)$ and $r$ appears in $w_2$. Since $\omega_2$ is a geodesic word on $A(t)$, by Theorem \ref{thm:centralizers}(iii), $r\neq s$ and $r\neq t$ and $r\in \Link(t)$. Since $r$ appears in $w_2$, a subword of $u_1$, we have that $r\in \Link(s)$,  and we found a triangle, which leads to contradiction.

(ii) We first prove the {\it only if} part. By (U2), if there is a third forbidden subword $u_3$, it must  contain a suffix of $u_1$ longer than $(s)$ (but shorter than $u_1$) and a prefix of $u_2$ longer than $(s)$ (but shorter than $u_2$). Then, since $w_2$ is a proper subword of $u_1$ and $u_2$, $u_1\not\equiv(s,s)$ and $u_2\not\equiv(s,s)$.

We now prove the {\it if} part. As $w_2$ is non-empty, if $w_2$ has length greater than 1, then $u_1\equiv(\dots, a_{t,s},s)$, $u_2\equiv(s,a_{t,s},\dots)$, and 
\begin{equation}\label{eq:intersec2} 
w\equiv(s,\dots, \lefteqn{\overbrace{\phantom{t,s,\dots,s,t,\dots,t}}^{a_{t,s}}}t,s,\dots,
\underbrace{s,t,\dots,t,s,\dots, t,s}_{a_{s,t}},\dots, t)
\end{equation}
and, as suggested by \eqref{eq:intersec2}, there is third forbidden subword. 
\end{proof}

\begin{Def}\label{Def:seq}
For a sequence $(u_i)_{i=1}^m$ in $U$ we define the following properties
\begin{enumerate}
\item[(R1)] if $u_i\equiv (s,\dots,s)$ and $u_{i+1}\equiv(s,\dots, s)$, then $u_i$ or $u_{i+1}$ is $(s,s)$.
\item[(R2)] if $u_i\equiv (s,\dots,s)$,  $u_{i+1}\equiv(t,\dots, t)$ and $s\neq t$ then $u_i\equiv(s,\dots,a_{t,s},s)$ and $u_{i+1}\equiv(t,a_{s,t},\dots,t)$. 
\item[(R3)] if $u_{i-1}\equiv(t,\dots,t),$ $u_i\equiv (s,\dots,s)$,  $u_{i+1}\equiv(t,\dots, t)$ and $s\neq t$ then $|u_i|> 2+ 2|a_{t,s}|$. \end{enumerate}
\end{Def}
In the next two lemmas we will show that $(w,(u_i)_{i=1}^m,(z_i)_{i=1}^{m-1})$ is a rigid chain if and only if  $(u_i)_{i=1}^m$ satisfy (R1), (R2) and (R3).

\begin{Lem}\label{Lem:rigidsequence1}
Let $(G,S)$ be triangle-free, even Coxeter system.  
Let $(u_i)_{i=1}^m$ be a sequence in $U$ satisfying {\rm (R1),(R2)} and {\rm (R3)}. Then there exits a unique  chain  having $(u_i)_{i=1}^m$ as sequence of forbidden words. Moreover, this chain is rigid.
\end{Lem}
\begin{proof}
By Lemma  \ref{lem:intersect}  a sequence $(u_i)_{i=1}^m$ in $U$ satisfing (R1) and (R2) also satisfies that, for all $i$, $u_i$ and $u_{i+1}$ has a unique non-empty intersection $z_i$ such that $z_i$ is a proper subword of $u_i$ and $u_{i+1}$, and $u_{i}*_{z_i}u_{i+1}$ contains only $u_i$ and $u_{i+1}$ as forbidden subwords. Then the sequence $(z_i)_{i=1}^{m-1}$ is uniquely determined.

Since $z_1$ is a proper subword of $u_1$,  $z_1\neq u_1$. To check that $|z_i|+|z_{i+1}|\leq|u_{i+1}|$, suppose that $u_{i-1}\equiv(r,\dots,r)$, $u_{i}\equiv(s,\dots,s)$ and $u_{i+1}\equiv(t,\dots,t)$. 

If $u_i\equiv(s,s)$, then $r=s$, $t=s$, $z_{i-1}=(s)$ is the first letter of $u_i$ and $z_{i}=(s)$ the last letter of $u_i$. Then $|z_{i-1}|+|z_{i}|\leq|u_{i}|$.

If $u_i\not\equiv(s,s)$, and  $r=s$,  by (R1) and  Lemma  \ref{lem:intersect}, $z_{i-1}\equiv(s)$ is the first letter of $u_i$ and $z_{i}$ is a proper subword of $u_i$, then $|z_{i-1}|+|z_{i}|\leq|u_{i}|$.

If $u_i\not\equiv(s,s)$, and  $t=s$, we argue as in the previous paragraph.

If $u_i\not\equiv(s,s)$, and  $r\neq s$ and $t\neq s$  by (R2) and  Lemma  \ref{lem:intersect}, $z_{i-1}\equiv(a_{r,s},r)$ and $z_{i}\equiv(t,a_{s,t})$. If $r\neq t$, then $u_i\equiv (s, a_{r,s},\dots, a_{t,s},s)$ and then $|z_{i-1}|+|z_{i}|\leq|u_{i}|$. If $t=r$, the result follows from (R3).

This completes the proof that there is a unique chain having $(u_i)_{i=1}^{m}$ as sequence of forbidden subwords. We left to the reader to check that this chain is rigid.
\end{proof}

We now prove the converse.
\begin{Lem}\label{Lem:rigidsequence2}
Let $(G,S)$ be a triangle-free, even Coxeter system.  
If a chain $(w,(u_i)_{i=1}^{m},(z_i)_{i=1}^{m-1})$ of rank $m$ is a rigid then $(u_i)_{i=1}^m$ satisfies {\rm (R1), (R2)} and {\rm (R3)}. 
\end{Lem}
\begin{proof}

We proceed by induction. By definition chains of rank $1$ are rigid chains. By Lemma \ref{lem:intersect}, chains of rank 2 are rigid if and only they satisfy (R1) and (R2) and $w\equiv u_1*_{z_1}u_2$. 

Suppose that by induction we have proven the result for chains of rank $m-1$.  If $(w,(u_i)_{i=1}^{m},(z_i)_{i=1}^{m-1})$ is rigid then 
\begin{enumerate}
\item the prefix of $w$ containing $(u_i)_{i=1}^{m_1}$ is rigid, which, by induction, is equivalent to saying that $(u_i)_{i=1}^{m-1}$ satisfies (R1),(R2) and (R3);
\item the suffix $v_1$ of $w$ following $u_{m-2}$ contains no forbidden subword except $u_m$. 

Assume that $u_{m-1}\equiv(s,\dots, s).$ 

By the induction hypothesis,  there is no forbidden subword in the amalgamation of $u_{m-2}*_{z_{m-2}}u_{m-1}$ but $u_{m-2}$ and $u_{m-1}$, and then, by Lemma \ref{lem:intersect}, either $z_{m-2}\equiv(s) $ or $z_{m-2}\equiv(s,a_{t,s})$. In both cases $(s,u'_{m-1})\in U$, where  $u'_{m-1}$ is the biggest suffix of $u_{m-1}$ not containing $z_{m-1}$. Let $u''_{m-1}\equiv(s,u'_{m-1})$. Since we are in a chain,   $u''_{m-1}$ and $u_m$ have non-trivial intersection $z_{m-1}$ and the amalgamation $u''_{m-1}*_{z_{m-1}}u_{m}$ contains no forbidden subword except for $u''_{m-1}$ and $u_m$ (since the chain is rigid).

Then by Lemma \ref{lem:intersect}, $u''_{m-1},u_m$ satisfy (R1) and (R2), and in particular $u_{m-1},u_m$ satisfy (R1) and (R2) and $z_{m-1}$ is unique.

By the induction hypothesis $(u_i)_{i=1}^{m-1}$ satisfy (R3). It only remains to show that  if $u_{m-2}\equiv(t,\dots,t)$ and $u_{m}\equiv(t,\dots,t)$, $t\neq s$, then $|u_{m-1}|>2+2|a_{t,s}|$.  We have shown that $z_{m-1}$ and $z_{m-2}$ are determined by $(u_i)_{i=1}^m$ and in this case $z_{m-2}\equiv(a_{s,t},t)\equiv(s,a_{t,s})\equiv z_{m-1}$. Since $(w,(u_i)_{i=1}^{m},(z_i)_{i=1}^{m-1})$ is a chain $|z_{m-2}|+|z_{m-1}|\leq |u_{m-1}|$. As a forbidden word, $u_{m-1}$ can not be equal to  $(s,a_{t,s},a_{t,s},s)$ and hence  $|z_{m-2}|+|z_{m-1}|< |u_{m-1}|$. This implies that $(u_i)_{i=1}^{m}$ satisfies (R3).
\end{enumerate}
\end{proof}


A labelled graph $\Gamma$ is {\it star-regular} if for every pair of vertices $u,v\in V\Gamma$, there is a isomorphism  of labelled graphs between $\Star(v)$ and $\Star(v).$
The main result of this section is the following:

\begin{Thm} \label{ThmevenCoxeter}
Let $(G,S)$ be an even Coxeter system with graph $\Gamma.$
Suppose that $\Gamma$ is triangle-free,  star-regular.
Then the geodesic growth of $G$ only depends on $|S|$ and the isomorphism class of the star of the vertices.

In particular, if $(G_1,S_1)$ and $(G_2,S_2)$ are triangle-free, star-regular, even Coxeter systems with $|S_1|=|S_2|$ and $\Star(v)\cong \Star(u)$, $v\in V\Gamma_1$, $u\in V\Gamma_2$, then $G_1$ and $G_2$ have the same geodesic growth. 
\end{Thm}
\begin{proof}
Let $U=U(G,S)$ be the set of forbidden words, and $|Q^{(m)}_n|$ denote the number of rigid chains of length $n$ and rank $m$. According to Theorem \ref{thm:criterium}, we have to show that the $|Q^{(m)}_n|$ only depend on $|S|$ and $\Star(v)$. By Lemmas \ref{Lem:rigidsequence1} and \ref{Lem:rigidsequence2}, the rigid chains are codified by sequences of words in $U$ satisfying (R1), (R2) and (R3).

We modify $\ga$ by adding two loops to each vertex $s$ of $V\ga$ and  labelling them by $1_s$ and $0_s$. Denote this graph by $\widehat{\ga}.$ Then $\widehat{\ga}$ is still star-regular. The key idea of the proof is that the number of labelled paths $\gamma$ in $\widehat{\ga}$  of a given length that do not have two consequetive edges labelled by $1$ depends only on $|S|$ and $\Star(v)$.


Suppose that $\gamma=(s_1,e_1,s_2,\dots,e_{m-1}, s_{m}),$ is a path where, for $i=1,\dots, m,$ $s_i$ is a vertex of $\widehat{\ga}$, $e_1,\dots,e_{m-1}$  are edges of  $\widehat{\ga}$ satisfying that for $i=1,\dots, m-2$, $e_i, e_{i+1}$ are not both labelled by 1.

We construct a sequence of forbidden words $(u_i)_{i=1}^{m}$, $u_i\equiv(s_i,p_i,\omega_i,q_i,s_i)$, associated to $\gamma$ with the following conditions on $p_i, \omega_i$ and $q_i$ to assure that $(u_i)_{i=1}^m$ satisfy (R1), (R2) and (R3). We will assume that $e_0$ and $e_{m}$ are labelled by $0$.
\begin{enumerate}
\item[(a)] If $e_i$ is labelled by $0$, then $p_i,\omega_i,q_i$ are  empty words, i.e $u_i=(s_i,s_i)$.
\item[(b)] If $e_i$ is not labelled by $0$:
	\begin{itemize}
	\item if $e_{i-1}$ is labelled by $0$, $p_i$ is the empty word; if $e_{i-1}$ is not labelled by $0$, then, since there can not be two consecutive edges labelled by $1$, we have that $s_{i-1}\neq s_i$ and we put $p_i\equiv a_{s_{i-1},s_i}.$
	\item if  $e_{i+1}$ is labelled by $0$,  $q_i$ is the empty word; if $e_{i+1}$ is not labelled by $0$, then, since there can not be two consecutive edges labelled by $1$, we have that $s_{i+1}\neq s_i$ and we put $q_i \equiv a_{s_{i+1},s_i}.$
	\end{itemize}
\item[(c)] If $e_i$ is not labelled by $0$, $\omega_i$ is a geodesic word in $A(s_i)$, such that $(p_i,\omega_i,q_i)$ is geodesic and if $s_{i-1}=s_{i+1}\neq s_i$, then $\omega_i$ is non-empty.
\end{enumerate}
Notice that (a) and (c) together with the fact that there are no two consecutive edges labelled 1, ensures that  $(u_i)_{i=1}^m$ satisfy (R1). Condition (b) and (c) ensures that  $(u_i)_{i=1}^m$ satisfy (R2). Condition (c) ensures (R3).

We notice that if we have a sequence of forbidden words $(u_i)_{i=1}^m$ satisfying (R1), (R2) and (R3), then it determines a path $\gamma=(s_1,e_1,\dots, e_{m-1},s_{m})$ in $\widehat{\ga}$ where $s_i$ is the first letter of $u_i$ and $e_i$ is and edge between $s_i$ and $s_{i+1}$ and if $s_i=s_{i+1}$, $e_i$ is labelled by $0$ if and only if  $u_{i-1}\equiv (s_i,s_i)$.

To show that the number of rigid chains of rank $m$ and length $n$ only depends on $|S|$ and $\Star(v)$, observe that for a rigid chains of rank $m$ and length $n$ are sequences  $(u_i)_{i=1}^m$ satisfying (R1), (R2) and (R3), and the rank and the length of the chain only depends on the $(u_i)_{i=1}^m$. In turn, such a sequences, depend on the paths described above and the choices for $\omega_i$ ($p_i$ and $q_i$ are fixed by the  labels of the edges in the paths). Since the graph is triangle-free, by Theorem \ref{thm:centralizers} (ii) the lengths of the $\omega_i$ as words over $S$ only depend on the length of $\omega_i$ as a geodesic word over $A(s)$ and the length of the $a_{t,s}$ $s,t\in S$. Since the graph is star-regular, the isomorphism from $\Star(t)$ to $\Star(s)$ induces a  bijection between elements of $A(s)$ and $A(t)$ that preserves the length as word over $S$ and that extends to a bijection from geodesic words over $A(s)$ to geodesic words over $A(t)$. Now the theorem follows because,  the number of labelled paths of a given length not having two consecutive ones only depend on $|S|$ and $\Star(v)$, and the possible lengths of $\omega_i$ only depend on the labels on the edges of the path.

\end{proof}

\begin{Ex}
Let $p,q$ be positive integers. Let $\ga_1$ be the graph consisting of 2 squares, with edges labelled alternatively $2p$ and $2q$. Let $\ga_2$ be the graph consisting of an octagon with edges labelled alternatively $2p$ and $2q$. Then the Coxeter groups defined by $\ga_1$ and $\ga_2$ are non-isomorphic (see, \cite[Theorem 5.1]{Bahls}) and have the same geodesic growth.
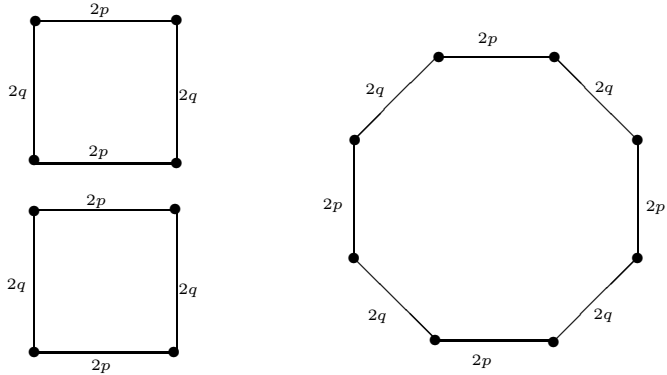
\begin{figure}[h]
\begin{center}
\setlength{\unitlength}{3947sp}%
\begingroup\makeatletter\ifx\SetFigFont\undefined%
\gdef\SetFigFont#1#2#3#4#5{%
  \reset@font\fontsize{#1}{#2pt}%
  \fontfamily{#3}\fontseries{#4}\fontshape{#5}%
  \selectfont}%
\fi\endgroup%
\begin{picture}(4041,2375)(841,-3983)
\put(863,-2215){\makebox(0,0)[lb]{\smash{{\SetFigFont{6}{7.2}{\rmdefault}{\mddefault}{\updefault}{\color[rgb]{0,0,0}$2q$}%
}}}}
{\color[rgb]{0,0,0}\thinlines
\put(3545,-3753){\circle*{74}}
}%
{\color[rgb]{0,0,0}\put(4816,-3234){\circle*{74}}
}%
{\color[rgb]{0,0,0}\put(4816,-2491){\circle*{74}}
}%
{\color[rgb]{0,0,0}\put(3033,-3234){\circle*{74}}
}%
{\color[rgb]{0,0,0}\put(3568,-1971){\circle*{74}}
}%
{\color[rgb]{0,0,0}\put(4296,-1971){\circle*{74}}
}%
{\color[rgb]{0,0,0}\put(3040,-2491){\circle*{74}}
}%
{\color[rgb]{0,0,0}\put(3033,-2489){\line( 0,-1){739}}
\put(3031,-3228){\line( 1,-1){523}}
\put(3553,-3752){\line( 1, 0){738}}
\put(4291,-3754){\line( 1, 1){523.500}}
\put(4816,-3232){\line( 0, 1){739}}
\put(4817,-2493){\line(-1, 1){522.500}}
\put(4296,-1969){\line(-1, 0){739}}
\put(3557,-1968){\line(-1,-1){522.500}}
}%
\put(3776,-3909){\makebox(0,0)[lb]{\smash{{\SetFigFont{6}{7.2}{\rmdefault}{\mddefault}{\updefault}{\color[rgb]{0,0,0}$2p$}%
}}}}
\put(4540,-3620){\makebox(0,0)[lb]{\smash{{\SetFigFont{6}{7.2}{\rmdefault}{\mddefault}{\updefault}{\color[rgb]{0,0,0}$2q$}%
}}}}
\put(3813,-1881){\makebox(0,0)[lb]{\smash{{\SetFigFont{6}{7.2}{\rmdefault}{\mddefault}{\updefault}{\color[rgb]{0,0,0}$2p$}%
}}}}
\put(4867,-2936){\makebox(0,0)[lb]{\smash{{\SetFigFont{6}{7.2}{\rmdefault}{\mddefault}{\updefault}{\color[rgb]{0,0,0}$2p$}%
}}}}
\put(2839,-2929){\makebox(0,0)[lb]{\smash{{\SetFigFont{6}{7.2}{\rmdefault}{\mddefault}{\updefault}{\color[rgb]{0,0,0}$2p$}%
}}}}
\put(3122,-3626){\makebox(0,0)[lb]{\smash{{\SetFigFont{6}{7.2}{\rmdefault}{\mddefault}{\updefault}{\color[rgb]{0,0,0}$2q$}%
}}}}
\put(3107,-2201){\makebox(0,0)[lb]{\smash{{\SetFigFont{6}{7.2}{\rmdefault}{\mddefault}{\updefault}{\color[rgb]{0,0,0}$2q$}%
}}}}
\put(4548,-2193){\makebox(0,0)[lb]{\smash{{\SetFigFont{6}{7.2}{\rmdefault}{\mddefault}{\updefault}{\color[rgb]{0,0,0}$2q$}%
}}}}
{\color[rgb]{0,0,0}\put(1904,-3827){\circle*{74}}
}%
{\color[rgb]{0,0,0}\put(1027,-3827){\circle*{74}}
}%
{\color[rgb]{0,0,0}\put(1911,-2929){\circle*{74}}
}%
{\color[rgb]{0,0,0}\put(1027,-2936){\circle*{74}}
}%
{\color[rgb]{0,0,0}\put(1919,-2639){\circle*{74}}
}%
{\color[rgb]{0,0,0}\put(1919,-1740){\circle*{74}}
}%
{\color[rgb]{0,0,0}\put(1027,-2616){\circle*{74}}
}%
{\color[rgb]{0,0,0}\put(1035,-1746){\circle*{74}}
}%
{\color[rgb]{0,0,0}\put(1027,-2639){\framebox(892,891){}}
}%
{\color[rgb]{0,0,0}\put(1027,-3827){\framebox(892,891){}}
}%
\put(1376,-1703){\makebox(0,0)[lb]{\smash{{\SetFigFont{6}{7.2}{\rmdefault}{\mddefault}{\updefault}{\color[rgb]{0,0,0}$2p$}%
}}}}
\put(1368,-2595){\makebox(0,0)[lb]{\smash{{\SetFigFont{6}{7.2}{\rmdefault}{\mddefault}{\updefault}{\color[rgb]{0,0,0}$2p$}%
}}}}
\put(1354,-2899){\makebox(0,0)[lb]{\smash{{\SetFigFont{6}{7.2}{\rmdefault}{\mddefault}{\updefault}{\color[rgb]{0,0,0}$2p$}%
}}}}
\put(1383,-3939){\makebox(0,0)[lb]{\smash{{\SetFigFont{6}{7.2}{\rmdefault}{\mddefault}{\updefault}{\color[rgb]{0,0,0}$2p$}%
}}}}
\put(1933,-3463){\makebox(0,0)[lb]{\smash{{\SetFigFont{6}{7.2}{\rmdefault}{\mddefault}{\updefault}{\color[rgb]{0,0,0}$2q$}%
}}}}
\put(856,-3426){\makebox(0,0)[lb]{\smash{{\SetFigFont{6}{7.2}{\rmdefault}{\mddefault}{\updefault}{\color[rgb]{0,0,0}$2q$}%
}}}}
\put(1933,-2238){\makebox(0,0)[lb]{\smash{{\SetFigFont{6}{7.2}{\rmdefault}{\mddefault}{\updefault}{\color[rgb]{0,0,0}$2q$}%
}}}}
{\color[rgb]{0,0,0}\put(4288,-3761){\circle*{74}}
}%
\end{picture}%

\end{center}
\caption{Two even Coxeter groups with the same geodesic growth.}
\end{figure}
\end{Ex}

Theorem \ref{ThmevenCoxeter} gives a generalization of our criterium for ragc's with the same geodesic growth. However, there are still many questions left regarding when two Coxeter groups have the same geodesic growth.

In this section the `eveness' and `triangle-freeness' were used to have particulary nice centralizers of generators, but there are no reasons for this condition to be necessary.  It is natural to ask the following:  Given a Coxeter system $(G,S)$ with $\ga(G,S)$ star-regular, does the geodesic growth of $G$ only depend on $|S|$ and the isomorphism class of $\Star(v)$?

It is worth mentioning that it is relevant in this context to consider the relation with the automatic generating set for racg's described in Section \ref{relation}. The greedy generating set \cite{Scott} generalizes the automatic generating an it is natural to expect that for an even Coxeter system $(G,S)$ with $\ga(G,S)$ star-regular, the spherical growth with respect to the greedy generating set of $G$ only depends on $|S|$ and the isomorphism class of $\Star(v)$.

\bigskip
\appendix
\section{Appendix: Centralizers  in even Coxeter groups}\label{s:centralizers}

In this section we give the promised proof of Theorem \ref{thm:centralizers}. As we mention, this is basically the proof of  \cite[Theorem 4.10]{BahlsMihalik}. We need two well-known facts, the fist one is a useful corollary of Lemma \ref{lem:abelianized}:
\begin{Cor}\label{cor:length}
Let $(G,S)$ be a Coxeter system, $w\in G$  and $s\in S$. Then $|ws|_S=|w|_S\pm 1$.
\end{Cor}

For a proof of the second fact see for example \cite[Proposition 1.1.1]{BjornerBrenti}.
\begin{Lem}\label{lem:prefix}
Let $(G,S)$ be a Coxeter sytem with relation matrix $(m_{s,t}).$
Let $s,t\in S$, $s\neq t.$ The order of $st$ in $G$ is $m_{s,t}$.
In particular, for all $t,s\in S$, the prefix of $(st)^{m_{s,t}}$ of length $m_{s,t}/2-1$ is geodesic in $S$.
\end{Lem}

The proof of Theorem \ref{thm:centralizers} is mostly based in the following lemma.

\begin{Lem}\label{lem:fact1}   
Let $(G,S)$ be an even Coxeter system and let $s\in S$.
Let $a_1,\dots, a_n\in A(s)$, suppose that $(a_1,\dots,a_n)$ is geodesic as a word on $A(s)$, and let $t\in \Link(s)$. 
\begin{enumerate}
\item[\normalfont{(a)}] If $(a_1,\dots,a_n,t)$ is  geodesic as a word on $S,$ then $(a_1,\dots,a_n,a_{t,s})$ is also geodesic on $S$.
\item[\normalfont{(b)}] If $(a_1,\dots, a_n,t)$ is not a geodesic as a word on $S,$ then there exists $1\leq i\leq n$ such that  $t$ commutes with $a_{i}\cdots a_n$.
\end{enumerate}
\end{Lem}

\begin{proof}
(a) First note that, by Lemma \ref{lem:prefix}, the statement is true for $n=0$. That is, for every $t\in \Link(s)$,  $a_{t,s}$ is  geodesic as a word on $S$. 

Suppose that $n$ is the length of the minimal counterexample of (a), we will derive a contradiction. That is, $(a_1,\dots,a_n,t)$ is  geodesic as a word on $S,$ but $(a_1,\dots,a_n,a_{t,s})$ is not. Then, by minimality, the word $\word \equiv(a_2,\dots,a_n,a_{t,s})$ is geodesic. Moreover, if $(w,s)\equiv (a_2,\dots,a_{t,s},s)$ is not geodesic, by the Deletion Condition, the last $s$ must cancel with some $s$ in $(a_2, \dots, a_n)$, but in this case $|\pi(\word s)|<|\pi(\word) |-1$, which contradicts Corollary \ref{cor:length}.

Let $b$ be the shortest suffix of $a_1$ such that $\word_1\equiv(b,a_2,\dots,a_n,a_{t,s},s)$ is not geodesic as word on $S$. Let $c$ be the largest prefix of $(a_{t,s},s)$ such that $(b,a_2,\dots,a_n,c)$ is geodesic. Since $(a_1,\dots, a_n,t)$ is geodesic, $c$ is non-empty. Also $c\not\equiv (a_{t,s},s).$ Let $d$ be the the prefix of $a_{t,s}$ of length $|c|+1$. Then by Corollary \ref{cor:forbidden}, $(b,a_2,\dots,a_n,d)$ is not geodesic but we obtain a geodesic word by deleting the first letter of $b$ and the last letter of $d$, which moreover are the same. Since $(b,a_2,\dots,a_n,d)$ is not geodesic, but  $\word$ is, $d\equiv(a_{t,s},s)$, and the first letter of $b$ is $s$. But $s$ commutes with all the $a_i$, and hence, $(b,s)$ is not geodesic and $(a_1,s)$ is not geodesic, which lead to a contradiction.

(b)  Since $(a_1,\dots, a_{n},t)$ is not geodesic as a word on $S$, by the Deletion Condition, there is $a_i$, $i\leq n$ and a letter in $a_i\equiv (t_i,s,\dots,s,t_i)$ which cancels with $t$. By Lemma \ref{lem:abelianized}, $t_i=t$, and thus $a_i\equiv a_{t,s}$. By Corollary \ref{cor:length}, $|\pi((a_i,\dots, a_n,t))|=|\pi((a_i,\dots, a_n))|-1,$ and hence $t$ must cancel either with the first or the last letter of $a_i$.  We will show that it cancels with the last. If $|a_i|=1$ the result is true. Assume that $|a_i|\geq 3.$ Since $(s,a_i,\dots,a_n)$ is also geodesic, if $t$ cancels with the first $t$ of $a_i$, then $|\pi((s,a_i,\dots,a_n,t))|\leq |\pi((s,a_i,\dots,a_n))|-3$, a contradiction. Thus $t$ cancels with the last letter of $a_i$, that is $t a_{i+1}\cdots a_n=a_{i+1}\cdots a_n t$.
\end{proof}

\begin{Lem}\label{lem:fact3}
Let $(G,S)$ be an even Coxeter system, $s\in S$ and  $a_1,\dots, ,a_{n+1}\in A(s)$. Suppose that $(a_1,\dots,a_n)$ is geodesic as a word in $S$ but  $(a_1,\dots,a_{n+1})$ is not. Then there exists $i\leq n$ such that $a_{n+1}=a_i$ and $a_{n+1}$ commutes with $a_{i+1}\cdots a_n$. 
\end{Lem}
\begin{proof}
Let $t$ be the first letter in $a_{n+1}$. By Lemma \ref{lem:fact1}(a) $(a_1,\dots, a_{n},t)$ is not geodesic. By  \ref{lem:fact1}(b), there exits $i\leq n$ such that $t$ commutes with $a_i\cdots a_n$. Since  $s a_{i+1}\cdots a_n=a_{i+1}\cdots a_n s$ the lemma follows.
\end{proof}

\begin{proof}[Proof of Theorem \ref{thm:centralizers}]
(i) is a restatement of  Brink's result in terms of generators, and it can be found in \cite[Theorem 2.6]{Bahls}. As noticed in  \cite[Theorem 4.10]{BahlsMihalik} this also follows  easily from Lemma \ref{lem:fact1}, Lemma \ref{lem:fact3} and the Deletion Condition.

(ii) is the content of the unpublished  \cite[Theorem 4.10]{BahlsMihalik} and we prove it here for completeness.

Clearly $s$ commutes with all elements of $A(s)$. Then a geodesic word on $A(s)\cup \{s\}$ has at most one $s$, and it can be at any position. We can restrict to show that  a geodesic word on $A(s)$ is also geodesic as word on $S$. 
 By Lemma \ref{lem:prefix} any word of length $1$ in $A(s)$ is geodesic on $S$. We argue by induction, suppose that $\word\equiv(a_1,\dots,a_n)$ is  geodesic as a word on $A(s)$, then by induction hypothesis $(a_1,\dots,a_{n-1})$ is geodesic as a word on $S$. If $\word$ is not geodesic as a word on $S$, by Lemma \ref{lem:fact3}, $\word$ is not geodesic on $A(s),$ a contradiction. This completes the proof of (ii).

(iii) By (ii) and the Deletion condition $(C(s),A(s)\cup\{s\})$ is a Coxeter system. We only have to show that  for all $t,r\in \Link(s)$, $r\neq s$ and  all $n\geq 1$ $(a_{t,s}a_{r,s})^n\neq 1$. We argue by contradiction, i.e. there exist some $t,r,n$ such that $(a_{t,s}a_{r,s})^n=1,$ in particular, there exits $m$ such that $(a_{t,s}a_{r,s})^m$ is geodesic in $S$.

If $(a_{t,s}a_{r,s})^ma_{t,s}$ is not geodesic as a word on $S$, then by Lemma \ref{lem:fact1}(a) $(a_{t,s}a_{r,s})^mt$ is not geodesic. By Lemma \ref{lem:fact1}(b) there exists some suffix $(a_{i,s},\dots,a_{r,s})$ that commutes with $t$. Then $(a_{i,s},\dots,a_{r,s})$ is a geodesic word representing some element in $C(t)$, in particular, by (i), $\supp((a_{i,s},\dots,a_{r,s}))\subseteq \Star(t).$ But since $\Star(s)$ spans a tree in $\ga$, $r\notin \Star(t)$, a contradicton.

If $(a_{t,s}a_{r,s})^ma_{t,s}$ is  geodesic in $S$ but $(a_{t,s}a_{r,s})^{m+1}$ is not, a similar argument implies that there is an edge joining $r,t$ in $\ga$, a contradiction. 
\end{proof}

\section*{Acknowledgments}

\footnotesize

The authors would like to thank John Meier and Jon McCammond for early pointers regarding the main question discussed in this paper, and would like to thank Sa\v sa Radomirovic and Claas R\"over for helpful discussions.

Both authors were partially supported by the Marie Curie Reintegration Grant 230889. The first-named author is also supported by  MCI (Spain) through
project MTM2011-25955 and EPSRC through project EP/H032428/1.

\bigskip

\textsc{Y. Antolin, University of Southampton,
Building 54, 
Highfield, Southampton,
SO17 1BJ, United Kingdom}

\emph{E-mail address}{:\;\;}\texttt{Y.Antolin-Pichel@soton.ac.uk}

\medskip

\textsc{L. Ciobanu,
Mathematics Department,
University of Fribourg,
Chemin du Muse\'e 23,
CH-1700 Fribourg, Switzerland
}

\emph{E-mail address}{:\;\;}\texttt{laura.ciobanu@unifr.ch}

\end{document}